\documentclass[11pt,a4paper]{amsart}
\usepackage[latin1]{inputenc}
\usepackage[english]{babel}
\usepackage{amsmath,amssymb,amsthm}
\usepackage[dvipsnames]{xcolor}
\usepackage{systeme,mathtools}
\usepackage{enumerate}
\usepackage{cite}
\usepackage{mathtools}

 \theoremstyle{plain}
\newtheorem{theorem}{Theorem}[section]
\newtheorem{lemma}[theorem]{Lemma}
\newtheorem{corollary}[theorem]{Corollary}
\newtheorem{prop}[theorem]{Proposition}

\newtheorem{proposition}[theorem]{Proposition}

\theoremstyle{definition}
\newtheorem{definition}[theorem]{Definition}
\newtheorem{example}[theorem]{Example}

\newtheorem{remark}[theorem]{Remark}

\let\phi\varphi
\let\epsilon\varepsilon
\newcommand{\norm}[1]{\Vert#1\Vert}
\newcommand{\bignorm}[1]{\bigl\Vert#1\bigr\Vert}
\newcommand{\Bignorm}[1]{\Bigl\Vert#1\Bigr\Vert}

\newcommand{\A}{\mbox{${\mathcal A}$}}

   \newcommand{\E}{\mbox{${\mathcal E}$}}
   \newcommand{\F}{\mbox{${\mathcal F}$}}
   
   \renewcommand{\H}{\mbox{${\mathcal H}$}}

    \newcommand{\M}{\mbox{${\mathcal M}$}}
   \newcommand{\N}{\mbox{${\mathcal N}$}}

\newcommand{\Cdb}{\mbox{${\mathbb C}$}}

\begin{document}
\title{$\ell^1$-contractive maps on noncommutative $L^p$-spaces}

\author[C. Le Merdy]{Christian Le Merdy}

\address{Laboratoire de Math\'ematiques de Besan\c{c}on,
Universite Bourgogne Franche-Comt\'e, France}
\email{\texttt{clemerdy@univ-fcomte.fr}}

\author[S. Zadeh]{Safoura Zadeh}
\address{
Department of Mathematics, Federal University of Para\'{i}ba, Brazil 
\& Faculty of Graduate Studies, Dalhousie University, Canada.
}
\email{\texttt{jsafoora@gmail.com}}



\begin{abstract}
Let $T\colon L^p(\M)\to L^p(\N)$ be a bounded operator between 
two noncommutative $L^p$-spaces, $1\leq p<\infty$. 
We say that $T$ is $\ell^1$-bounded (resp.
$\ell^1$-contractive) if $T\otimes I_{\ell^1}$ extends to a bounded 
(resp. contractive) map from 
$L^p(\M;\ell^1)$ into $L^p(\N;\ell^1)$. 
We show that Yeadon's factorization theorem for $L^p$-isometries, 
$1\leq p\not=2 <\infty$, applies to an isometry $T\colon L^2(\M)\to L^2(\N)$
if and only if $T$ is $\ell^1$-contractive. We also show that 
a contractive operator $T\colon L^p(\M)\to L^p(\N)$ is automatically
$\ell^1$-contractive if it satisfies one of the following two 
conditions: either $T$ is $2$-positive; or $T$ is separating, that is,
for any disjoint $a,b\in L^p(\M)$ (i.e. $a^*b=ab^*=0)$, the images 
$T(a),T(b)$ are disjoint as well.
\end{abstract}

\date{\today}

\maketitle

\section{Introduction}\label{sec1}
Let $\M$ and $\N$ be two semifinite von Neumann algebras. 
For any $1\leq p<\infty$, consider the associated noncommutative $L^p$-spaces
$L^p(\M)$ and $L^p(\N)$. A remarkable theorem of Yeadon \cite{Y} 
(see Theorem \ref{yeadon} below) asserts that if $p\not=2$ and $T\colon L^p(\M)\to L^p(\N)$ is a linear isometry, 
then there exist a normal Jordan homomorphism 
$J\colon \M\to\N$, a positive operator $B$ 
affiliated with $\N$ and a partial isometry $w\in\N$ such that
$w^*wB=B$, $J(a)$ commutes with $B$ for all
$a\in \M$,
and
\begin{equation}\label{eqY0}
T(a) = wBJ(a),
\end{equation}
for all $a\in \M\cap L^p(\M)$.

This striking factorization property is the noncommutative version of the celebrated description
 of isometries on classical (=commutative) $L^p$-spaces due to Banach \cite{Ban} and Lamperti \cite{Lam}. 
We refer to the books \cite{FJ1} and \cite{FJ2} for details on these results, complements  and historical background.

The work presented in this paper was originally motivated by the following question, concerning the case $p=2$:
what are the linear isometries  $T\colon L^2(\M)\to L^2(\N)$ which admit a Yeadon type factorization, that is, 
isometries for which
there exist $J,B,w$ as above such that (\ref{eqY0}) holds true for any 
$a\in \M\cap L^2(\M)$?

This issue leads us to introduce a new property, called $\ell^1$-boundedness, which is defined as follows. 
Consider the $\ell^1$-valued noncommutative $L^p$-space  $L^p(\M;\ell^1)$ introduced by Junge \cite{J} (see also  \cite{P5} and \cite{JX}).
Let $T\colon L^p(\M)\to L^p(\N)$ be a bounded operator. We say that $T$ is $\ell^1$-bounded if $T\otimes I_{\ell^1}$ extends to a bounded map 
$$
T\overline{\otimes} I_{\ell^1}\colon
L^p(\M;\ell^1)\longrightarrow L^p(\N;\ell^1).
$$
We further say that $T$ is $\ell^1$-contractive if the
map $T\overline{\otimes} I_{\ell^1}$ is a contraction.
The main result of this paper (Theorem \ref{main1} below) is that an isometry $T\colon L^2(\M)\to L^2(\N)$ is $\ell^1$-contractive if and only if it admits a Yeadon type factorization.

To explain the relevance of Theorem \ref{main1} we note that
$\ell^1$-boundedness is a noncommutative analogue of regularity for maps acting on commutative $L^p$-spaces. (We refer to \cite[Chapter 1]{P4} for definitions and background on regular maps.)
It follows that 
Theorem \ref{main1} is a noncommutative extension of the well-known result stating that a linear isometry between commutative $L^2$-spaces is a Lamperti operator if and only if it is contratively regular, if and only if it is a subpositive contraction (see e.g \cite{Pel}).

The proof of Yeadon's theorem heavily relies on the fact that for $p\not=2$, any linear isometry 
$T\colon L^p(\M)\to L^p(\N)$ has the following property: if $a,b\in L^p(M)$ are disjoint, that is $a^*b=ab^*=0$, then $T(a)$ and $T(b)$ are disjoint as well. Such maps are called separating in the present paper. We show that a bounded operator $L^p(\M)\to L^p(\N)$ is separating if and only if it admits a Yeadon type factorization.

The concept of $\ell^1$-boundedness is interesting in its own sake and this paper aims at studying some of its main features. We show in particular that a contractive operator 
$T\colon L^p(\M)\to L^p(\N)$ is automatically
$\ell^1$-contractive either if $T$ is separating (see Theorem \ref{main0}) or if $T$ is $2$-positive (see Proposition \ref{6.2}).

\section{Notion of $\ell^1$-boundedness and background} \label{sec2}

In this section, we provide some
background on noncommutative $L^p$-spaces and on the $\ell^1$-valued
spaces $L^p(\M;\ell^1)$. Then we introduce the notions
of $\ell^1$-boundedness and $\ell^1$-contractivity and establish some preliminary results.

Let $\mathcal{M}$ be a semifinite von Neumann algebra equipped with a normal semifinite faithful trace $\tau_{\tiny{\M}}$. We briefly recall the noncommutative $L^p$-spaces $L^p(\mathcal{M})$, 
$0< p\leq\infty$, associated with $(\mathcal{M},\tau_{\tiny{\M}})$ 
and some of their basic properties. 
The reader is referred to the survey \cite{PX} and references therein for details and further properties.

If $\mathcal{M}\subset B(\mathcal{H})$ acts on some Hilbert space $\mathcal{H}$,
the elements of $L^p(\mathcal{M})$ can be viewed as closed densely defined
(possibly unbounded) operators on $\mathcal{H}$. More precisely, let $\mathcal{M}^{\prime}$
denote the commutant of $\mathcal{M}$ in $B(\mathcal{H})$. A closed densely defined operator $a$ is said to be affiliated with $\mathcal{M}$ if $a$ commutes with every unitary of $\mathcal{M}^{\prime}$. An affiliated operator $a$ is called measurable (with respect to $(\mathcal{M},\tau_{\tiny{\M}})$) if there is a positive number $\lambda>0$ such that $\tau_{\tiny{\M}}(\epsilon_\lambda)<\infty$, where $\epsilon_\lambda=\chi_{[\lambda,\infty)}(\lvert a\rvert)$ is the projection associated with the indicator function of $[\lambda,\infty)$ in the Borel functional calculus of $\lvert a\rvert$. Then the set $L^0(\mathcal{M})$ of all measurable operators forms a $*$-algebra (see e.g. \cite[Chapter I]{Te} for a proof and also for the definitions 
of algebraic operations on $L^0(\mathcal{M})$). We proceed with defining $L^p(\mathcal{M})$ as a subspace of  $L^0(\mathcal{M})$. First note that for any $a\in L^0(\mathcal{M})$ and any $0<p<\infty$, the operator $\lvert a\rvert^p=(a^* a)^{\frac{p}{2}}$ belongs to $L^0(\mathcal{M})$. 
If $L^0(\mathcal{M})^+$ denotes the positive cone of $L^0(\mathcal{M})$, that is the set of all positive operators in $L^0(\mathcal{M})$, the trace $\tau_{\tiny{\M}}$ extends to a positive tracial functional on $L^0(\mathcal{M})^+$, taking values in $[0,\infty]$, also denoted by $\tau_{\tiny{\M}}$. For any $0< p<\infty$, the noncommutative $L^p$-space, $L^p(\mathcal{M})$, associated with $(\mathcal{M},\tau_{\tiny{\M}})$, is
$$
L^p(\mathcal{M}):=\bigl\{a\in L^0(\mathcal{M}):\tau_{\tiny\M}(\lvert a\rvert^p)<\infty\bigr\}.
$$
For $a\in L^p(\mathcal{M})$, let $\|a\|_p:=\tau_{\tiny\M}(\left\lvert a\rvert^p\right)^{\frac{1}{p}}$. For $1\leq p<\infty$, $\|\cdot\|_p$ defines a complete norm, and for $p<1$, a complete
$p$-norm. We let $L^\infty(\mathcal{M}):=\mathcal{M}$, equipped with its operator norm $\|\,\cdotp\|_\infty$.

For any $0<p\leq\infty$ and any $a\in L^p(\mathcal{M})$, the adjoint operator $a^*$ 
belongs to $L^p(\mathcal{M})$ and $\|a^*\|_p=\|a\|_p$. Furthermore, 
we have that $a^* a\in L^\frac{p}{2}(\mathcal{M}) $ and $\lvert a\rvert\in L^p(\mathcal{M}) $, with $\|\lvert a\rvert\|_p=\|a\|_p$. 
More generally, for any $0<p,q,r\leq\infty$ 
with $\frac{1}{p}+\frac{1}{q}=\frac{1}{r}$, we have that
$ab\in L^r(\M)$ if $a\in L^p(\M)$ and $b\in L^q(\M)$, with 
H\"{o}lder's inequality
\begin{equation}\label{Holder}
\|ab\|_r\leq\|a\|_p\|b\|_q.
\end{equation}

For any $1\leq p<\infty$, let $p^\prime:=\frac{p}{p-1}$ be the conjugate number of $p$. 
Then by (\ref{Holder}), $ab$ belongs to $L^1(\M)$
for any $a\in L^p(\mathcal{M})$ and $b\in L^{p^{\prime}}(\mathcal{M})$.
Further the duality pairing  
$$
\langle a,b\rangle=\tau_{\tiny{\M}}(ab),\qquad  a\in L^p(\mathcal{M}),\ b\in L^{p^{\prime}}(\mathcal{M}),
$$
yields an isometric isomorphism $L^{p}(\mathcal{M})^*= L^{p^{\prime}}(\mathcal{M})$.
In particular, we may identify $L^1(\mathcal{M})$ with the (unique) predual $\mathcal{M}_*$ of $\mathcal{M}$. 
These duality results will be used without further reference in the paper.

We let $L^p(\mathcal{M})^+:= L^0(\mathcal{M})^+\cap L^p(\mathcal{M})$ denote the positive cone of $L^p(\mathcal{M})$. 
A bounded operator $T\colon L^p(\mathcal{M})\to L^p(\mathcal{N})$ between
two noncommutative $L^p$-spaces is called positive if it maps   $L^p(\mathcal{M})^+$ into $L^p(\mathcal{\N})^+$ .

If $\mathcal{M}=B(\mathcal{H})$, the algebra of all bounded operators on $\mathcal{H}$, 
and $\tau_{\tiny\M}=tr$, the usual trace on $B(\mathcal{H})$, then the associated 
noncommutative $L^p$-space is the Schatten class
$S^p(\mathcal{H})$. If $\M=L^\infty(\Omega,\F,\mu)$ is the commutative
von Neumann algebra associated with a measure space
$(\Omega,\F,\mu)$, then $L^p(\M)$ coincides with the classical $L^p$-space 
$L^p(\Omega,\F,\mu)$.

Let $tr$ denote the usual trace on $B(\ell^2)$ and consider the von Neumann 
algebra tensor product $B(\ell^2)\overline{\otimes}\M$,
equipped with the normal semifinite faithful trace $tr\overline{\otimes} 
\tau_{\tiny\M}$ (see \cite[Chapter V, Proposition 2.14]{Ta}). 
Any element of $L^p(B(\ell^2)\overline{\otimes}\M)$ can be regarded as 
an infinite matrix $(b_{ij})_{i,j\geq1}$, 
with $b_{ij}\in L^p(\mathcal{M})$. We let $L^p(\M;\ell^2_c)$ denote 
the subspace of $L^p(B(\ell^2) \overline{\otimes}\M)$ consisting of all matrices 
whose entries off the first column are all zero. We 
regard this space as a sequence space by
saying that a sequence $(b_n)_{n\geq1}$ of $L^p(\M)$ belongs to $L^p(\M;\ell^2_c)$ if the infinite matrix 
$$
\begin{pmatrix}
b_1&0&\dots&0&\dots\\
\vdots&\vdots&&\vdots&\\
b_n&0&\dots&0&\dots\\
\vdots&\vdots&&\vdots&
\end{pmatrix}
$$
represents an element of $L^p(B(\ell^2)\overline{\otimes}\M)$. Similarly, we define $L^p(\M;\ell^2_r)$ as the
subspace of $L^p(B(\ell^2) \overline{\otimes}\M)$ consisting of all matrices whose entries off the first row are all zero.

We let $E_{ij}$, $i,j\geq1$, denote the usual matrix units of $B(\ell^2)$, 
and regard $S^p(\ell^2)\otimes L^p(\M)$ as a subspace of $L^p(B(\ell^2)\overline{\otimes}\M)$ 
in the usual way. For any finitely supported sequence $(a_n)_{n\geq1}$ and
$(b_n)_{n\geq1}$ of $L^p(\M)$, we have
\begin{equation}\label{SF}
\Bignorm{\sum_{n=1}^{\infty} E_{n1}\otimes 
b_n}_{L^p(\tiny{\M};\ell^2_c)}=
\Bignorm{\sum_{n=1}^{\infty}  
b_n^* b_n}_{\frac{p}{2}}^{\frac12}
\end{equation}
and
\begin{equation}\label{SFbis}
\Bignorm{\sum_{n=1}^{\infty} E_{1n}\otimes 
a_n}_{L^p(\tiny{\M};\ell^2_r)}=
\Bignorm{\sum_{n=1}^{\infty}  
a_n  a_n^*}_{\frac{p}{2}}^{\frac12}.
\end{equation}

When $1\leq p<\infty$,
elements of $L^p(\M;\ell^2_c)$ and $L^p(\M;\ell^2_r)$ can be approximated by finitely supported sequences,
thanks to the following (easy) result.

\begin{lemma}\label{2.3}
Let $1\leq p<\infty$ and suppose that $(b_n)_{n\geq1}$ is a 
sequence in $L^p(\mathcal{M})$. The following are equivalent:
\begin{itemize}
\item[(i)] $(b_n)_{n\geq1}$ belongs to $L^p(\mathcal{M};\ell_c^2)$.
\item[(ii)] There exists a positive constant $K$ such that for every 
$N\geq1$, 
$$
\Bignorm{\sum_{n=1}^N E_{n1}\otimes b_n}_{L^p(\tiny{\M};\ell^2_c)} \leq K.
$$
\item[(iii)] The series 
$\sum_{n\geq1} E_{n1} \otimes b_n$ 
converges in $L^p(\mathcal{M};\ell^2_c)$.
\end{itemize}
Moreover, the same result holds with $\ell^2_c$ 
replaced by $\ell^2_r$ and $E_{n1}$ replaced by $E_{1n}$. 
\end{lemma}

\begin{remark}\label{rem}
By (\ref{SF}) and the Cauchy convergence test, we see
that a sequence $(b_n)_{n\geq 1}$ in $L^p(\mathcal{M})$ satisfies
the assertion (iii) of Lemma \ref{2.3} if and only if
the series 
$\sum_n b_n^*b_n$ converges in $L^{\frac{p}{2}}(\M)$. In this case, 
the identity (\ref{SF}) holds true for $(b_n)_{n\geq 1}$.
\end{remark}

Let $1\leq p<\infty$.
In \cite{J}, Junge defined $L^p(\mathcal{M};\ell^1)$ as the space of all sequences 
$x=(x_n)_{n\geq1}$ in $L^p(\mathcal{M})$ for which there 
exist families $(u_{kn})_{k,n\geq 1}$,
$(v_{kn})_{k,n\geq 1}\in L^{2p}(\mathcal{M})$ and a positive constant $K$ such that 
\begin{equation}\label{eqnew}
\Bignorm{\sum_{k,n=1}^N u_{kn}u_{kn}^{*}}_p \leq K
\qquad\hbox{and}\qquad
\Bignorm{\sum_{k,n=1}^N v^*_{kn}v_{kn}}_p\leq K 
\end{equation}
for any $N\geq1$, and $\sum_{k =0}^{\infty}
u_{kn}v_{kn}=x_n$,
for all $n\geq 1$. (The convergence of the series is ensured
by (\ref{eqnew}) and Lemma \ref{2.3}.)
He showed that this a Banach space when
equipped with the norm 
$$
\bignorm{(x_n)_{n\geq 1}}_{L^p(\mathcal{M};\ell^1)}=
\inf\left\{\sup_N
\Bignorm{\sum_{k,n=1}^N u_{kn}u_{kn}^*}_p^{\frac12} 
\sup_N\Bignorm{\sum_{k,n=1}^N v^*_{kn}v_{kn}}_p^{\frac12}\right\},
$$
where the infimum is taken over all families 
$(u_{kn})_{k,n\geq 1}$ and
$(v_{kn})_{k,n\geq 1}$ as above.

The following alternative description is well-known to
specialists (and implicit in \cite[pp. 537-538]{P1}).
We give a proof for the sake of completeness.

\begin{lemma}\label{2.6}
Suppose that $1\leq p<\infty$ and that $(x_n)_{n\geq1}$ is a sequence 
in $L^p(\mathcal{M})$. Then the following are equivalent:
\begin{itemize}
\item[(i)] $(x_n)_{n\geq1}$ belongs to $L^p(\mathcal{M};\ell^1)$
and $\norm{(x_n)_{n\geq 1}}_{L^p(\mathcal{M};\ell^1)}<1$.
\item[(ii)] There exist sequences $(a_n)_{n\geq1}$ and $(b_n)_{n\geq1}$ 
in $L^{2p}(\mathcal{M})$ such that $a_nb_n=x_n$ for all $n\geq 1$,
the series $\sum_n a_n  a_n^*$ and $\sum_n b_n^* b_n$
converge in $L^p(\M)$, and we have
$$
\Bignorm{\sum_{n=1}^\infty a_n  a_n^*}_p <1
\qquad\hbox{and}\qquad
\Bignorm{\sum_{n=1}^\infty b_n^* b_n}_p< 1.
$$
\end{itemize}
\end{lemma}

\begin{proof}
The assertion ``(ii) $\Rightarrow$ (i)" is obvious. 
Conversely assume
(i) and  consider $u_{kn}, v_{kn}$
in $L^{2p}(\M)$ 
satisfying (\ref{eqnew}) for some $K<1$, and 
such that 
$$
x_n=\sum_{k=1}^{\infty}u_{kn}v_{kn},\qquad\text{for all}\ n\geq 1.
$$

We regard $L^p(\M)$ as a subspace of 
$L^p(B(\ell^2)\overline{\otimes}\M)$ by identifying
any $b\in L^p(\M)$ with $E_{11}\otimes b$.
We set
$$
u_n:=(u_{kn})_{k\geq1}\in L^{2p}(\mathcal{M};\ell^2_r)
\qquad\hbox{and}\qquad 
v_n:=(v_{kn})_{k\geq1}\in L^{2p}(\mathcal{M};\ell^2_c)
$$
for all $n\geq 1$, that we regard as elements of
$L^{2p}(B(\ell^2)\overline{\otimes}\M)$. Then 
$x_n= u_nv_n$ for all $n\geq 1$.

By polar decomposition, 
there exist $\varphi_n\in L^\infty(\mathcal{M};\ell^2_r)$ and 
$\psi_n\in L^\infty(\mathcal{M};\ell^2_c)$ such that 
$$
\norm{\phi_n}_\infty\leq 1,\quad
\norm{\psi_n}_\infty\leq 1,\quad
u_n= \vert u_n^*\vert\varphi_n
\quad\hbox{and}\quad
v_n=\psi_n\lvert v_n\rvert.
$$
If we let $a_n=\lvert u_n^*\rvert$ and 
$b_n=\phi_n\psi_n\lvert v_n\rvert$, then we have 
$$
x_n=u_nv_n
=\lvert u_n^*\rvert\phi_n\psi_n\lvert v_n\rvert 
=a_nb_n
$$
for all $n\geq1$. Further $a_n,b_n$ belong to $L^{2p}(\M)$.
Next we have
$$
a_n a_n^* =\lvert u_n^*\rvert^2=u_n u_n^*=\sum_{k=1}^\infty u_{kn} u_{kn}^*,
$$
hence for any $N\geq 1$, $\norm{\sum_{n=1}^N a_na_n^*}_p\leq K$.
By Lemma \ref{2.3} and Remark \ref{rem}, this implies the 
convergence of the series of the $a_na_n^*$ in $L^p(\M)$, with 
$\bignorm{\sum_{n=1}^\infty a_na_n^*}_p\leq K$.
Likewise, since $(\phi_n\psi_n)^* (\phi_n\psi_n)\leq1 $, we have
\begin{align*}
b_n^* b_n&=\lvert v_n\rvert (\phi_n\psi_n)^* (\phi_n\psi_n) \lvert v_n\rvert\\
&\leq\lvert v_n\rvert^2=v_n^\ast v_n=
\sum_{k=1}^\infty v_{kn} ^\ast v_{kn},
\end{align*}
from which we deduce that the series of the $b_n^*b_n$ 
converges in $L^p(\M)$, with 
$\bignorm{\sum_{n=1}^\infty b_n^*b_n}_p\leq K$.
This proves (ii).
\end{proof}

When dealing with positive sequences, the study of the $L^p(\M,\ell^1)$-norm is simple.
We learnt the following result from \cite{Xu}.

\begin{lemma}\label{Xu}
Let $1\leq p<\infty$, let $(x_n)_{n\geq 1}$ be a sequence 
of $L^p(\M)$ and assume that $x_n\geq 0$ for any $n\geq 1$.
The following are equivalent.
\begin{itemize}
\item [(i)] $(x_n)_{n\geq 1}$ belongs to $L^p(\M;\ell^1)$.
\item [(ii)] The series $\sum_n x_n\,$ converges
in $L^p(\M)$.
\end{itemize}
Further in this case, we have
\begin{equation}\label{Xu1}
\bignorm{(x_n)_{n\geq 1}}_{L^p(\tiny\M;\ell^1)}\, 
=\, \Bignorm{\sum_{n=1}^{\infty} x_n}_p.
\end{equation}
\end{lemma}

\begin{proof}
It follows from (\ref{SF}) and (\ref{SFbis})
that for any finitely supported families $(a_n)_{n\geq 1}$
and $(b_n)_{n\geq 1}$ in $L^{2p}(\M)$, we have
$$
\Bignorm{\sum_{n=1}^{\infty} a_n b_n}_p
\leq \Bignorm{\sum_{n=1}^{\infty} a_n a_n^*}_p^{\frac12}
\Bignorm{\sum_{n=1}^{\infty} b_n^*b_n}_p^{\frac12}.
$$
The assertion ``$(i) \Rightarrow (ii)$" and the inequality
$\geq$ in (\ref{Xu1}) follow at once (here we do not need
any positivity 
assumption on the $x_n$).

Assume conversely
that the series $\sum_n x_n\,$ converges
in $L^p(\M)$ and set $a_n=b_n=x_n^{\frac12}$. 
Then the convergence of $\sum_n a_na_n^*$
and  $\sum_n b_n^* b_n$ are trivial and
$x_n=a_nb_n$ for all $n\geq 1$. This 
implies (i), as well as the inequality
$\leq$ in (\ref{Xu1}).
\end{proof}

We remark that for any $c=(c_n)_{n\geq 1}\in\ell^1$
and any $x\in L^p(\M)$, the sequence $(c_nx)_{n\geq 1}$
belongs to $L^p(\M;\ell^1)$. Further the mapping $x\otimes c\mapsto 
(c_nx)_{n\geq 1}$ extends to an embedding 
$$
L^p(\mathcal{M})\otimes \ell^1\subset L^p(\mathcal{M};\ell^1)
$$
and with this convention, $L^p(\mathcal{M})\otimes \ell^1$ is a dense subspace of $L^p(\mathcal{M};\ell^1)$.

Let $(e_k)_{k\geq 1}$ denote the canonical basis of $\ell^1$. Then 
we let $L^p(\M;\ell^1_2)$ be the direct sum $L^p(\M)\oplus L^p(\M)$ 
equipped with the norm $\norm{(x,y)}
=\norm{x\otimes e_1 + y\otimes e_2}_{L^p(\tiny{\M};\ell^1)}$,
for any
$x,y\in L^p(\M)$. 

Throughout the paper we will consider two semifinite
von Neumann algebras $\M$ and $\N$ equipped with 
normal semifinite faithful traces $\tau_{\tiny{\M}}$
and $\tau_{\tiny{\N}}$, respectively, and we will consider various 
bounded operators $L^p(\M)\to L^p(\N)$, for $1\leq p<\infty$.

\begin{definition}
We say that a  bounded operator $T\colon 
L^p(\mathcal{M})\to L^p(\mathcal{N})$ is 
\begin{enumerate}[(i)]

\item $\ell^1$-bounded if 
$T\otimes I_{\ell^1}$ extends to a bounded map 
$$
T\overline{\otimes}I_{\ell^1}\colon L^p(\mathcal{M};\ell^1)
\longrightarrow L^p(\mathcal{N};\ell^1).
$$
In this case, the norm of $T\overline{\otimes}I_{\ell^1}$ is called 
the $\ell^1$-bounded norm of $T$ and is denoted by $\norm{T}_{\ell^1}$;

\item $\ell^1$-contractive if it is $\ell^1$-bounded and $\norm{T}_{\ell^1}
\leq 1$;

\item $\ell^1_2$-contractive if for every $x,y\in L^p(\M)$, we have
$$
\|\left(T(x),T(y)\right)\|_{L^p(\tiny{\N};\ell^1_2)}
\leq\|(x,y)\|_{L^p\left(\tiny{\M};\ell^1_2\right)}.
$$
\end{enumerate}
\end{definition}

\begin{remark}
In the case $p=1$, we note that
$L^1(\M;\ell^1) \simeq\ell^1(L^1(\M))$ isometrically. This implies that any bounded 
operator $T\colon
L^1(\M)\to L^1(\N)$ is automatically 
$\ell^1$-bounded, with $\|T\|_{\ell^1}=\|T\|$.
\end{remark}

In the rest of this section, 
we compare $\ell^1$-boundedness with Pisier's notion
of complete regularity. 
Let us recall that for a hyperfinite von Neumann algebra $\M$ and an operator space $E$, 
Pisier \cite[Chapter 3]{P5} introduced a vector valued noncommutative $L^p$-space $L^p(\M)[E]$.
Let ${\rm Max}(\ell^1)$ be $\ell^1$ equipped with its so-called maximal 
operator space structure (see e.g. \cite[Chapter 3]{P3}). 
It turns out that 
\begin{equation}\label{Max}
L^p(M;\ell^1) \simeq L^p(M)[{\rm Max}(\ell^1)]\qquad \hbox{isometrically},
\end{equation}
when $\M$ is hyperfinite (see \cite{J,JX}).

Assume that the semifinite von Neumann algebras
$\mathcal{M}, \mathcal{N}$ are both hyperfinite.
Let $T\colon
L^p(\mathcal{M})\to L^p(\mathcal{N})$ be a bounded operator. Following Pisier \cite{P2}, $T$ is 
called completely regular if there 
exists a constant $K\geq 0$ such that for any $n\geq 1$,
$$
\bignorm{T\otimes I_{M_n}\colon L^p(\mathcal{M})[M_n]\longrightarrow  L^p(\mathcal{N})[M_n]}\,\leq K.
$$
In this case the least possible $K$ is denoted by $\norm{T}_{reg}$ and is called the completely regular
norm of $T$.
It is noticed in \cite{P2} that if $T$ is 
completely regular, then for any operator space 
$E$, $T\otimes I_E$ (uniquely) extends to a bounded operator $T\overline{\otimes} I_E$
from $L^p(\mathcal{M})[E]$ into $L^p(\mathcal{N})[E]$, with 
$$
\bignorm{T\overline{\otimes} I_E\colon L^p(\mathcal{M})[E]\longrightarrow  
L^p(\mathcal{N})[E]}\,\leq \norm{T}_{reg}.
$$
Combining this fact with (\ref{Max}), we obtain  that
if $T\colon L^p(\mathcal{M})\to L^p(\mathcal{N})$ is completely regular, then
$T$ is $\ell^1$-bounded, with $\norm{T}_{\ell^1}\leq \norm{T}_{reg}$.

The next example shows that the converse is wrong.

\begin{example}
We consider the specific case $\mathcal{M}=B(\ell^2)$, and we let
$T\colon S^p(\ell^2)\to S^p(\ell^2)$ be the transposition map.
This map is $\ell^1$-contractive. This is an easy
fact, which is a special case of Theorem \ref{main0} below. Here is a direct argument.

Let $(x_n)_{n\geq 1}$ be in $L^p(\M);\ell^1)$
and let $(a_n)_{n\geq 1}$ and $(b_n)_{n\geq 1}$ 
be two sequences
belonging to $L^{2p}(\M;\ell^2_r)$ and $L^{2p}(\M;\ell^2_c)$,
respectively, such that $x_n=a_nb_n$ for any $n\geq 1$.
Then 
$T(x_n)= {}^t b_n {}^t a_n$ for any $n\geq 1$,
$({}^t a_n)_{n\geq 1}$ belongs to $L^{2p}(\M;\ell^2_c)$,
$({}^t b_n)_{n\geq 1}$ belongs to $L^{2p}(\M;\ell^2_r)$ and
we both have
$$
\bignorm{({}^t a_n)_{n}}_{L^{2p}(\tiny{\M};\ell^2_c)}
=\bignorm{(a_n)_{n}}_{L^{2p}(\tiny{\M};\ell^2_r)}
\quad\hbox{and}\quad
\bignorm{({}^t b_n)_{n}}_{L^{2p}(\tiny{\M};\ell^2_r)}
=\bignorm{(b_n)_{n}}_{L^{2p}(\tiny{\M};\ell^2_c)}.
$$
The result follows at once.

Let us now prove that $T$ is not completely
regular. We need a little operator space technology,
in particular we use the Haagerup tensor product $\otimes_h$,
the operator spaces $R,C$ and the interpolation spaces
$R(\theta)=[C,R]_\theta$, $0\leq \theta\leq 1$, for which
we refer to \cite{P3}.

Let $(e_k)_{k\geq 1}$ be the canonical basis of $\ell^2$.
As it is outlined in \cite[Theorem 1.1 and p.20]{P5}, for any operator space $E$,
the mapping $e_i\otimes x\otimes e_j\mapsto E_{ij}\otimes x$,
for $i,j\geq 1$ and $x\in E$, uniquely extends to an isometric 
isomorphism
\begin{equation}\label{Pisier}
S^p[E]\,\simeq\, R\bigl(\tfrac{1}{p}\bigr) \otimes_h E
\otimes_h
R\bigl(1-\tfrac{1}{p}\bigr).
\end{equation}

Assume that $T$ is completely regular and let $K=\norm{T}_{reg}$.
Apply (\ref{Pisier}) with $E=R$. For any $n\geq 1$, we have
$$
\bigl(T\otimes I_R\bigr)
\Bigl(\sum_{k=1}^n e_k\otimes e_k\otimes e_1\Bigr)
\,=\, \sum_{k=1}^n e_1\otimes e_k\otimes e_k,
$$
hence 
$$
\Bignorm{\sum_{k=1}^n e_k\otimes e_k}_{R\otimes_h
R\bigl(1-\tfrac{1}{p}\bigr)}\,\leq K\, 
\Bignorm{\sum_{k=1}^n e_k\otimes 
e_k}_{R\bigl(\frac{1}{p}\bigr)\otimes_h
R}.
$$
It follows from the calculations in \cite[Chapter 1]{P5}
(see also \cite{JLM}) that 
$$
R\bigl(\tfrac{1}{p}\bigr)\otimes_h
R\,\simeq\, S^{2p}
\qquad\hbox{and}\qquad
R\otimes_h
R\bigl(1-\tfrac{1}{p}\bigr)\,\simeq\, S^{(2p)'}
$$
isometrically, where $(2p)'$ is the conjugate number of $2p$.
Let $Q_n\in B(\ell^2)$ be the orthogonal 
projection onto ${\rm Span}\{e_1,\ldots,e_n\}$. Then 
$\sum_{k=1}^n e_k\otimes e_k = Q_n$ in
the above identifications. Hence 
we obtain that
$$
\norm{Q_n}_{(2p)'}\leq K\norm{Q_n}_{2p}.
$$
Since $\norm{Q_n}_{q}=n^{\frac{1}{q}}$ for any $1<q<\infty$,
we obtain that $n^{1-\frac{1}{2p}}\leq K
n^{\frac{1}{2p}}$, equivalently, $n\leq K n^{\frac{1}{p}}$, for any 
$n\geq 1$. This yields a contradiction if $p>1$. In the case 
$p=1$, the fact that $T$ is not completely regular 
is obtained by applying (\ref{Pisier}) with $E=C$ instead of $E=R$.
\end{example}

\begin{remark}\label{Comm}
Here we consider the commutative case. 
Let $(\Omega,\F,\mu)$ be a measure space. For any
operator space $E$, 
$L^p(L^\infty(\Omega))[E]$ coincides with the Bochner 
space $L^p(\Omega;E)$. Thus if $(\Omega',\F'\mu')$ 
is another measure space and $T\colon L^p(\Omega)\to 
L^p(\Omega')$ is any bounded operator, then $T$ is completely 
regular (in the above sense) 
if and only if $T$ is regular in the lattice sense (see \cite[Chapter 1]{P4}
for details and background). Moreover 
in this case, the completely 
regular norm of $T$ coincides with its regular norm. 
It follows from \cite[Paragraph 1.2]{P4} that $T$ is regular
if (and only if) $T\otimes I_{\ell^1}$ extends to a bounded
map $T\overline{\otimes} I_{\ell^1}$ from $L^p(\Omega;\ell^1)$
into $L^p(\Omega';\ell^1)$ and in this case, we have
$\norm{T\overline{\otimes} I_{\ell^1}}=\norm{T}_{reg}$. Consequently,
$T\colon L^p(\Omega)\to 
L^p(\Omega')$ is $\ell^1$-bounded if and only if $T$ is regular 
and in this case, the $\ell^1$-bounded norm of $T$ is 
equal to its regular norm.
\end{remark}

\section{Disjointness and separating operators} \label{sec3}

In \cite{K}, Kan introduced the concept of Lamperti operators on commutative $L^p$-spaces,
which include $L^p$-isometries, $1\leq p\neq2<\infty$, and positive $L^2$-isometries. 
He then proved a structural theorem for such operators. 
In this section we provide a noncommutative version of this result, as well as a connection with $\ell^1$-boundedness.

Let us first recall some facts related to Jordan homomorphisms that we require in this section. 
(We refer to \cite{S}, \cite{OS} and \cite[Exercices 10.5.21-10.5.31]{KR} for general information.)
A Jordan homomorphism between von Neumann algebras $\mathcal{M}$ and $\mathcal{N}$ is a 
linear map $J\colon\mathcal{M}\to\mathcal{N}$ that satisfies $J(a^2)=J(a)^2$ 
and $J(a^\ast)=J(a)^\ast$, for every $a\in\mathcal{M}$. 
We say that the Jordan homomorphism $J\colon\M\to\N$ is a Jordan monomorphism when 
$J$ is one-to-one. Any Jordan homomorphism is a positive contraction
and any Jordan monomorphism is an isometry.

Let $J \colon\mathcal{M}\to\mathcal{N} $ be a  
Jordan homomorphism and 
let $\mathcal{D}\subset \mathcal{N}$ be the von Neumann algebra 
generated by $J(\mathcal{M})$. Let $e:=J(1)$. 
Then $e$ is a projection and $e$ is the unit of $\mathcal{D}$. 
According to \cite[Theorem 3.3]{S} (see also \cite[Corollary 7.4.9.]{OS}), there exist 
projections $g$ and $f$ in the center of $\mathcal{D}$ such that
\begin{itemize}
\item [(i)] $g+f=e$;
\item [(ii)] $a\mapsto J(a)g$ is a $\ast$-homomorphism;
\item [(iii)] $a\mapsto J(a)f$ is an anti-$\ast$-homomorphism;
\end{itemize}
Let $\mathcal{N}_1=g\mathcal{N}g$ and $\mathcal{N}_2=f\mathcal{N}f$. 
We let $\pi\colon\mathcal{M}\to\mathcal{N}_1$
and $\sigma\colon\mathcal{M}\to\mathcal{N}_2$ be defined by $\pi(a)=J(a)g$ 
and $\sigma(a)=J(a)f$, for all $a\in\M$.
Then, $\mathcal{D}= \mathcal{D}_1\mathop{\oplus}\limits^{\infty}\mathcal{D}_2$ 
and $J(a)=\pi(a)+\sigma(a)$, for all $a\in\M$. We will use the suggestive notations 
\begin{equation}\label{pisigma}
J=\begin{pmatrix}
\pi&0\\
0&\sigma
\end{pmatrix}
\qquad\hbox{and}\qquad
J(a)=\begin{pmatrix}
\pi(a)&0\\
0&\sigma(a)
\end{pmatrix}
\end{equation}
to refer to such a central decomposition. We note that 
$J$ is normal (i.e. $w^*$-continuous) if and only if $\pi$ and $\sigma$ are normal.

Assume that $\M\subset B(\H)$ acts on some Hilbert space $\H$ and let 
$x$ be a closed densely defined operator on $\H$, affiliated with $\M$.
If $x$ is self-adjoint, with 
polar decomposition $x=w\lvert x\rvert$, we 
let $s(x)$ denote the projection $w^\ast w\ (=ww^\ast)$, called  
the support of $x$.

The following remarkable characterization of $L^p$-isometries, $1\leq p\not=2<\infty$,
is at the root of our investigations.

\begin{theorem}[Yeadon \cite{Y}]\label{yeadon}
For $1\leq p<\infty,\ p\neq 2$, a bounded operator
$$
T\colon L^p(\mathcal{M})\longrightarrow L^p(\mathcal{N})
$$ 
is an isometry if and only 
if there exist a normal Jordan monomorphism $J\colon\mathcal{M}\to\mathcal{N}$, 
a partial isometry $w\in\mathcal{N}$, and a positive operator 
$B$ affiliated with $\mathcal{N}$, which verify the following conditions:
\begin{itemize}
\item [(a)] $T(a)=wBJ(a)$ for all $a\in\mathcal{M}\cap L^p(\mathcal{M})$;
\item [(b)] $w^{\ast}w=J(1)=s(B)$;
\item [(c)] every spectral projection of $B$ commutes with $J(a)$, for all $a\in\mathcal{M}$;
\item [(d)] $\tau_{\mathcal{M}}(a)=\tau_{\mathcal{N}}(B^pJ(a))$ for all $a\in\mathcal{M}^+$.
\end{itemize}
\end{theorem}

This motivates the introduction of the following concept. Since we would like to consider maps $T:L^p(\M)\to L^p(\N)$ that are not necessarily isometries we drop part $(d)$ of Theorem \ref{yeadon} in our definition.

\begin{definition}\label{YTF}
We say that a bounded operator $T\colon L^p(\mathcal{M})\to L^p(\mathcal{N})$, $1\leq p\leq\infty$, 
has a {\bf ``Yeadon type factorization''} if there exist a normal Jordan homomorphism 
$J\colon \mathcal{M}\to\mathcal{N}$, a partial isometry $w\in\mathcal{N}$, and a positive 
operator $B$ affiliated with $\mathcal{N}$, which verify the following conditions:
\begin{itemize}
\item [(a)] $T(a)=wBJ(a)$ for all $a\in\mathcal{M}\cap L^p(\mathcal{M})$;
\item [(b)] $w^{\ast}w=J(1)=s(B)$;
\item [(c)] every spectral projection of $B$ commutes with $J(a)$, for all $a\in\mathcal{M}$.
\end{itemize}
\end{definition}

\begin{remark}
The argument in the last paragraph of the proof of \cite[Theorem 2]{Y} 
shows that for an operator $T\colon L^p(\mathcal{M})\to L^p(\mathcal{N})$ 
with a Yeadon type factorization, 
$w$, $B$ and $J$ from Definition \ref{YTF} are uniquely determined by $T$. 
We call $(w, B, J)$ the Yeadon triple of the operator $T$.
\end{remark}

The crucial property that allowed Yeadon to describe $L^p$-isometries is that 
they map disjoint elements to disjoint elements. This property is shared by operators other than isometries and 
as we show in Proposition \ref{sep}, it characterizes operators with a Yeadon type factorization.
We introduce the relevant concepts and
supply a few preparatory results.

\begin{definition} Let $a$ and $b$ be elements in $L^p(\mathcal{M})$, $1\leq p
\leq \infty$. 
We say $a$ and $b$ are disjoint if $ab^*=a^*b=0$.
\end{definition}

\begin{lemma}\label{disjointness}
The elements $a$ and  $b$ in $L^p(\mathcal{M})$ are disjoint if and only if $\lvert a\rvert$ 
and $\lvert b\rvert$ are disjoint and $\lvert a^*\rvert$ and $\lvert b^*\rvert$ are disjoint.
\end{lemma}

\begin{proof}
Let $a,\ b\in L^p(\mathcal{M})$. First we note that $a$ and $b$ are disjoint if and only if
$$
\text{Im}(b)\subseteq \text{Im}(a)^\perp
\qquad\text{\ and\ }\qquad
\text{Ker}(a)^\perp\subseteq \text{Ker}(b).
$$
This implies that $\lvert a\rvert$ and $\lvert b \rvert$ are disjoint if 
and only if $\text{Ker}(a)^\perp = \text{Ker}(\lvert a\rvert)^\perp\subseteq 
\text{Ker}(\lvert b\rvert)= \text{Ker}(b)$. Then 
$\lvert a^*\rvert$ and $\lvert b^*\rvert$ are disjoint if and only if 
$\text{Ker}(a^*)^\perp\subset \text{Ker}(b^*)$, which is
itself equivalent to
$\text{Im}(b)\subseteq \text{Im}(a)^\perp$. 
Therefore,
$a$ and $b$ are disjoint if and only if $\lvert a\rvert$ and $\lvert b\rvert$ are disjoint and $\lvert a^*\rvert$ and $\lvert b^*\rvert$ are disjoint.
\end{proof}

\begin{remark}\label{Ortho}
Consider the special case $p=2$ and let $a,b$ be two positive elements in 
$L^2(\M)$. Then $a$ and $b$ are disjoint if (and only if) $\tau_{\tiny\M}(ab)=0$, that is,
if and only if $a$ and $b$ are orthogonal in the Hilbertian sense.
Indeed assume that $\tau_{\tiny\M}(ab)=0$. Then
$0=\tau_{\tiny\M}(ab)=\tau_{\tiny\M}((a^{\frac{1}{2}}b^{\frac{1}{2}})(b^{\frac{1}{2}} a^{\frac{1}{2}}))$ and
$a^{\frac{1}{2}}b^{\frac{1}{2}}$ is the adjoint of 
$b^{\frac{1}{2}} a^{\frac{1}{2}}$.
Since the trace $\tau_{\tiny\M}$ is faithful, 
this implies
that $a ^{\frac{1}{2}} b^{\frac{1}{2}}=0$.
Therefore, $ab=a ^{\frac{1}{2}}(a ^{\frac{1}{2}} b ^{\frac{1}{2}} )b ^{\frac{1}{2}} =0$.
Hence $a$ and $b$ are disjoint.
\end{remark}

\begin{definition}\label{dytf} 
We say that a bounded operator $T\colon L^p(\mathcal{M})\to L^p(\mathcal{N})$, $1\leq p\leq\infty$, 
is separating if whenever $a,b\in L^p(\M)$ are disjoint, then $T(a)$ and $T(b)$ are disjoint.
\end{definition}

\begin{lemma}\label{dj}
Any Jordan homomorphism $J\colon\mathcal{M}\to\mathcal{N}$ is separating.
\end{lemma}

\begin{proof}
Let $J\colon\mathcal{M}\to\mathcal{N}$ be a Jordan homomorphism and consider a decomposition
$
J=\begin{pmatrix}
\pi&0\\
0&\sigma
\end{pmatrix}
$ 
as in (\ref{pisigma}).

Suppose that $a$ and $b$ are disjoint elements of $\M$.
Then we have 
\begin{align*}
	J(a)^*J(b)&= \begin{pmatrix}
\pi(a)^*\pi(b)&0\\
0&\sigma(a)^*\sigma(b)
\end{pmatrix}\\
&= \begin{pmatrix}
\pi(a^*b)&0\\
0&\sigma(ba^*)
\end{pmatrix}=0.
\end{align*}
Similarly, we can show that $J(a)J(b)^*=0$, and therefore $J$ is separating.
\end{proof}

\begin{lemma}\label{dj2}
Suppose that a bounded operator $T\colon L^p(\mathcal{M})\to L^p(\mathcal{N})$, 
$1\leq p<\infty$, is separating on 
$\mathcal{M}\cap L^p(\mathcal{M})$, that is,
$T(a)$ and $T(b)$ are disjoint for any disjoint $a$ and $b$ in $\M\cap L^p(\M)$. Then $T$ is separating.
\end{lemma}

\begin{proof}
Let $a,b\in L^p(\mathcal{M})$ with $a^\ast b=ab^\ast=0$.
We let $a=v\lvert a\rvert$ and $b=w\lvert b\rvert$ be the polar decompositions of $a$ and $b$, respectively.
By Lemma \ref{disjointness}, we have $\lvert a\rvert\lvert b\rvert=0$. 

For any $n\geq 1$, let $p_n:=\chi_{[-n,n]}(\lvert a \rvert)$ be the projection associated with the 
indicator function of $[-n,n]$ in the Borel functional calculus of $\lvert a\rvert$, 
and similarly let $q_n:=\chi_{[-n,n]}(\lvert b \rvert)$. 
Let $a_n:=ap_n$ and $b_n:=bq_n$.
We have 
$$
p_n\lvert a\rvert=\lvert a\rvert p_n\to\lvert a\rvert
\qquad\hbox{and}\qquad
q_n\lvert b\rvert=\lvert b\rvert q_n\to\lvert b\rvert
$$ 
in $L^p(\mathcal{M})$. This implies that $a_n\to a$ and $b_n\to b$ in $L^p(\mathcal{M})$. 

Note that for any $n\geq1$,  
$a_n$ and $b_n$ belong to $\mathcal{M}\cap L^p(\mathcal{M})$.
Further we have
$$
a_n^\ast b_n=p_na^\ast bq_n=0
$$ 
and 
$$
a_n b_n^\ast =v\lvert a \rvert p_nq_n\lvert b\rvert w^\ast
=vp_n\lvert a\rvert\lvert b\rvert q_nw^\ast=0. 
$$
Thus $a_n$ and $b_n$ are disjoint.

By assumption this implies that 
$T(a_n)^\ast T(b_n)=0$ and $T(a_n)T(b_n)^\ast=0$. 
Passing to the limit, we deduce 
that $T(a)^\ast T(b)=0$ and $T(a)T(b)^\ast=0$. 
\end{proof}

From now on, we consider
$$
\E:=\bigl\{e\in\M\, :\, e\text{ is a projection and }\tau_{\tiny\M}(e)<\infty\bigr\}
\quad \text{ and }\quad
\A:=\bigcup_{\small{e\in\E}} e\M e.
$$
For any $x\in \M$, $exe\to x$ in the $w^*$-topology of $\M$, when $e\to 1$.
Further for any $1\leq p<\infty$ and any $x\in L^p(\M)$,
$exe\to x$ in $L^p(\M)$. Thus
$\A$ is a $w^*$-dense subspace of $\M$ and
a norm dense subspace of $L^p(\M)$, for any $1\leq p<\infty$.
Lemma \ref{C} below is a $w^*$-extension result of independent interest, which will be used 
in the proof of Proposition \ref{sep}.

Given any $w\in\M^\ast$ and $a,b\in\M$, we let $awb\in\M^\ast$ be defined by
$$
awb(x) : = w(bxa), \qquad \text{for all }x\in\M.
$$
Recall e.g. from  \cite[pp 126-128]{Ta}
the decomposition 
\begin{equation}\label{1-sum}
\M^\ast=\M_\ast\mathop{\oplus}\limits^1 \M^\ast_s,
\end{equation}
where $\M^\ast_s$ denotes the space of singular functionals on $\M$, 
and $\M_*$ is the predual of $\M$, which coincides with the space of normal functionals on $\M$.
It is well-known that if $w=w_n+w_s$ is the aforementioned decomposition of $w$, 
then $aw_n b$ and $aw_s b$ are the normal part and the singular part of $awb$, respectively.

\begin{lemma}\label{C}
Let $Y$ be a dual Banach space. For any bounded operator $u\colon \A\to Y$, the following are equivalent:
\begin{itemize}
\item[(i)] For every $e\in\E$, the restriction $u|_{e\tiny{\M} e}\colon e\M e\to Y$ is $w^*$-continuous.
\item[(ii)] There exists a $w^\ast$-continuons extension $\widehat{u}\colon \M\to Y$ of $u$.
\end{itemize} 
\end{lemma}

\begin{proof}
The implication ``$(ii)\Rightarrow(i)$" is trivial. 
For the converse, we assume (i).

We first  consider the case when $Y=\mathbb{C}$. 
Suppose that $\alpha\in\A^\ast$ is such that $\alpha|_{e\tiny{\M}e}$ is $w^*$-continuous 
for each $e\in\E$. Using Hahn-Banach, we let $w\colon\M\to\mathbb{C}$ be a bounded extension of 
$\alpha$ to $\M$ and we consider its decomposition $w=w_n+w_s$ according to (\ref{1-sum}).

For every $e\in\E$, $ewe$ is $w^*$-continuous. 
We noticed that $ewe=ew_n e+ew_s e$ is the decomposition of $ewe$. 
Since $ewe$ and $ew_n e$ are $w^*$-continuous, the singular part  $ew_s e$ of $ewe$ must be zero, 
and consequently, $ewe=ew_n e$, for every $e\in\E$. This implies that
the restriction of $w_n$ to $\A$ coincides with $\alpha$.
Thus $\widehat{\alpha}=w_n$ is a 
$w^\ast$-continuous extension of $\alpha$.

For the general case, let $v=u^\ast|_{Y_\ast}\colon Y_\ast\to\A^\ast$
be the restriction of the adjoint of $u$ to the predual of $Y$.
Let 
$$
\kappa\colon \M_\ast\longrightarrow\A^\ast
$$
denote the restriction map taking any $\nu\in\M_*$ to $\nu|_{\A}$. 
This is an isometry (by Kaplansky's theorem, say), whose range coincides 
with the space of all functionals $\A\to\mathbb{C}$ which admit a $w^*$-continuous extension to $\M$.

For each $\eta\in Y_\ast$, $\eta\circ u|_{e\tiny{\M} e}$ is $w^*$-continuous, for every $e\in\E$. 
By our argument for the case $Y=\mathbb{C}$, this implies 
that $\eta\circ u\in\text{Im}(\kappa)$.  This means that $v$ is valued in $\text{Im}(\kappa)$.
We can therefore consider $w=\kappa^{-1}\circ v\colon 
Y_\ast\to\M_\ast$ and define
$\widehat{u}=w^\ast\colon \M\to Y$. By construction, $\widehat{u}$ is $w^*$-continuous. 

We now claim that
$\widehat{u}$ is an extension of $u$. 
To see this, recall that for any $\eta\in Y_*$, the functional
$\kappa^{-1}\circ u^\ast(\eta)$ is an extension to $\M$ of $u^\ast(\eta)\colon\A\to\mathbb{C}$. Consequently,
$$
\langle \widehat{u}(exe),\eta\rangle=\langle exe,w(\eta)\rangle=
\langle exe,\kappa^{-1}\circ u^\ast(\eta)\rangle=\langle exe,u^\ast(\eta)\rangle=\langle u(exe),\eta\rangle
$$
for any $x\in\M$, any $e\in\E$ and any $\eta\in Y_*$. This proves the claim.
\end{proof}

\begin{prop}\label{sep}
For $1\leq p<\infty$, a bounded operator
$T\colon L^p(\mathcal{M})\to L^p(\mathcal{N})$ is separating if and only if it has a Yeadon type factorization.
\end{prop}

\begin{proof}
Assume that $T$ is separating. 
We adapt Yeadon's argument from \cite{Y},
taking into account that our operators are no longer necessarily isometries.

For any $e\in\E$, let $B_e=\vert T(e)\vert$ and let 
$T(e)=w_e B_e$ be the polar decomposition of $T(e)$. 
We have
\begin{equation}\label{eq00}
B_ew^*_ew_e=B_e=w_e^*w_eB_e.
\end{equation}
Set $J(e):=w_e^\ast w_e=s(B_e)$.  
If $e$ and $f$ are in $\E$ and $ef=0$, then since $T$ is separating we have 
\begin{equation}\label{eqsep}
T(e)^\ast T(f)=T(e)T(f)^\ast=0.
\end{equation}
Using (\ref{eqsep}), the argument in the proof of \cite[Theorem 2]{Y} shows that 
$J\colon \E\to\N$ extends to a linear map $J\colon\A\to\N$ such that 
\begin{equation}\label{Jordan}
J(a^2)=J(a)^2,\quad J(a^\ast)=J(a)^\ast,\quad\hbox{and}\quad \| J(a)\|_{\infty}\leq\|a\|_{\infty}
\end{equation}
for all $a\in\A$, 
\begin{equation}\label{eqY}
T(exe)=w_e B_e J(exe)	
\end{equation}
for all $e\in\E$ and $x\in\M$, and 
\begin{equation}\label{ef}
B_f J(e) = J(e)B_f = B_e,\quad\text{for any}\ e,f\in\E,\ \text{with}\ e\leq f.
\end{equation}

Note that by (\ref{Jordan}),
the restriction of $J$ to $e\M e$ is a Jordan homomorphism for any
$e\in\E$. Consequently,
\begin{equation}\label{JP}
J(exe) = J(e)J(x)J(e)
\end{equation}
for any $e\in\E$ and any $x\in\A$.

We now show that $J$ admits a normal extension (still denoted by)
$J\colon\M\to\N$. (Note that Yeadon's argument in the isometric case 
does not apply to our general case.) According to Lemma \ref{C} it suffices
to show that the restriction of $J$ to $e\M e$ is normal for any $e\in\E$.
To see this, we fix such an $e$ and we
let $(ex_ie)_i$ be a bounded net of $e\mathcal{M}e$ converging to $exe$ in the 
$w^*$-topology of $\M$. Then $ex_ie\to exe$ in the weak topology
of $L^p(\M)$, hence $w_e^*T(ex_ie)\to w_e^*T(exe)$ in the weak topology
of $L^p(\N)$. By (\ref{eqY}) and (\ref{eq00}), this implies that
$\tau_{\tiny\N}(AB_eJ(ex_ie))\to \tau_{\tiny\N}(AB_eJ(exe))$
for any $A\in L^{p'}(\N)$, where $p'$ is the conjugate number of $p$. 
Note that by (\ref{JP}),
the restriction of $J$ to $e\M e$ is valued 
in $J(e)\N J(e)$. To deduce from the above convergence property
that $J(ex_ie)\to J(exe)$ in the $w^*$-topology of
$\N$, it therefore suffices to check that 
$$
\bigl\{AB_e\, :\, A\in J(e)L^{p'}(\N)J(e)\bigr\}\quad\text{is 
a dense subset of}\ J(e)L^1(\N) J(e).
$$
This is indeed the case, since $B_e=\vert T(e)\vert\in L^p(\N)$
and $s(B_e) =J(e)$.

We note that since $J$ is $w^*$-continuous and $\A$ is $w^*$-dense in $\M$,
(\ref{JP}) holds true for any $e\in\E$ and any $x\in\M$.

We now use the increasing net $\E$ and we recall that $e\to 1$ in the 
$w^*$-topology of $\M$. Since $J$ is normal,  
$J(1)$ is the $w^*$-limit of $J(e)$.
Then using (\ref{eq00}) and (\ref{ef}),
the same argument as in \cite[Theorem 3.1]{JRS} 
can be implemented to obtain extensions $B$ 
(as supremum of the $B_e$) 
and $w$ (as strong limit of the $w_e$)
which satisfy properties $(b)$ and $(c)$ of Definition \ref{YTF}.
By (\ref{ef}) we further have
\begin{equation}\label{BJe}
BJ(e) = J(e)B = B_e\qquad\hbox{and}\qquad wJ(e)=w_e,
\end{equation}
for any $e\in\E$.

We now aim at showing property
$(a)$ of Definition \ref{YTF}.
For any $y\in\M^+\cap L^1(\M)$, using spectral projections,
we find a sequence $(e_n)_{n\geq 1}$ in $\E$ such that
$e_nye_n$ is increasing to $y$ when $n\to\infty$.
 Since
$J$ is normal, this implies that $B^pJ(e_nye_n)$ 
is increasing to $B^pJ(y)$ when $n\to\infty$. 
Consequently, using the normality of $\tau_{\tiny\N}$, we obtain that
\begin{equation}\label{Lim}
\tau_{\tiny\N}\bigl(B^pJ(y)\bigr) = \lim_{n\to\infty} 
\tau_{\tiny\N}\bigl(B^pJ(e_nye_n)\bigr).
\end{equation}

Let $x\in \M\cap L^p(\M)$.
Consider a decomposition 
$J=\begin{pmatrix}
\pi&0\\
0&\sigma
\end{pmatrix}
$ as in (\ref{pisigma}). Then we have
\begin{equation}\label{pi-sig}
J(\vert x\vert^p) =\begin{pmatrix}
\vert \pi(x)\vert^p &0\\
0&\vert \sigma(x^*)\vert^p
\end{pmatrix}.
\end{equation}
Then by a well-known argument 
(see the proof of the easy implication 
of \cite[Theorem 2]{Y}), this implies  that
\begin{equation}\label{NormY}
\norm{wBJ(x)}_p^p= \tau_{\tiny\N}\bigl(B^p
\vert J(x)\vert^p\bigr)=
\tau_{\tiny\N}\bigl(B^pJ(\vert x\vert^p)\bigr).
\end{equation}
Using (\ref{Lim}) with $y=\vert x\vert^p$, we deduce that
for some sequence $(e_n)_{n\geq 1}$ in $\E$, we have
$$
\norm{wBJ(x)}_p^p =\lim_{n\to\infty} 
\tau_{\tiny\N}\bigl(B^pJ(e_n\vert x\vert^p e_n)\bigr).
$$
Fix $e\in \E$. 
Combining (\ref{eqY}) and (\ref{JP}), we have
$T(exe) =  w_eB_eJ(x)J(e)$ from which we deduce 
as in (\ref{NormY}) that
$$
\norm{T(exe)}_p^p =\tau_{\tiny\N}\bigl(B_e^p\vert
J(x)\vert^p\bigr).
$$
Using (\ref{JP}) again, and (\ref{BJe}),  we have 
$$\tau_{\tiny\N}\bigl(B^pJ(e\vert x\vert^p e)\bigr)
=\tau_{\tiny\N}\bigl(B^pJ(e)J(\vert x\vert^p)J(e)\bigr)
=\tau_{\tiny\N}\bigl(B_e^pJ(\vert x\vert^p)\bigr). 
$$
Consequently, using (\ref{pi-sig}), we have
\begin{align*}
\tau_{\tiny\N}\bigl(B^pJ(e\vert x\vert^p e)\bigr)
&= \tau_{\tiny\N}\bigl(B_e^p
\vert \pi(x)\vert^p\bigr) + \tau_{\tiny\N}\bigl(B_e^p
\vert \sigma(x^*)\vert^p\bigr)\\
& \leq\tau_{\tiny\N}\bigl(B_e^p
\vert J(x)\vert^p\bigr) + \tau_{\tiny\N}\bigl(B_e^p
\vert J(x^*)\vert^p\bigr)\\
&\leq \norm{T(exe)}_p^p +\norm{T(ex^*e)}_p^p\\
&\leq 2\norm{T}^p\norm{x}_p^p.
\end{align*} 
Applying this with $e=e_n$ and passing to the limit, we 
deduce that $wBJ(x)\in L^p(\N)$ for any $x\in \M\cap L^p(\M)$, with
$\norm{wBJ(x)}_p\leq 2\norm{T}\norm{x}_p$.
This shows the existence of a (necessarily unique)
bounded operator $T'\colon L^p(\M)\to L^p(\N)$
such that $T'(x) = wBJ(x)$ for any $x\in \M\cap L^p(\M)$.
By (\ref{eqY}) and (\ref{BJe}), $T$ and $T'$ coincide on $\A$. Since the latter is
dense in $L^p(\M)$, we obtain that $T=T'$, hence 
property
$(a)$ of Definition \ref{YTF} is satisfied.

For the converse suppose that $T\colon
L^p(\mathcal{M})\to L^p(\mathcal{N})$ has a
Yeadon type factorization,
with Yeadon triple $(w,B,J)$, and
let us show that $T$ is separating. By Lemma \ref{dj2}, it is enough to show 
that $T$ is separating on $\M\cap L^p(\M)$. Let $a$ and $b$ be disjoint elements in $\M\cap L^p(\M)$, then
\begin{align*}
	T(a)^\ast T(b)&=J(a)^\ast B w^\ast w B J(b)&\text{by (a) in Definition \ref{YTF}}\\
	&=J(a)^\ast B^2 J(b)&\text{by (b) in Definition \ref{YTF}}\\
	&=B^2 J(a)^\ast J(b)&\text{by (c) in Definition \ref{YTF}}\\
	&=0&\text{by Lemma \ref{dj}}.
\end{align*}
Similarly we can show that $T(a)T(b)^\ast=0$, and 
hence $T(a)$ and $T(b)$ are disjoint. 
\end{proof}

After a first version of this paper
was circulated, we were informed that the `only if' part of Proposition \ref{sep}
was proved independently in \cite{HRW}.

We now give a series of remarks on this statement.

\begin{remark}\label{Restrict}
\

(a)\ In Proposition \ref{sep}, the proof that a separating map $T$ admits
a Yeadon type factorization only uses the separation property on positive
elements (even on projections with finite trace). 
Hence a bounded operator $T\colon L^p(\mathcal{M})\to L^p(\mathcal{N})$ is separating
if and only if for any $a,b\in L^p(\M)^+$, $ab=0$ implies that 
$T(a)^*T(b)=T(a)T(b)^*=0$.

\smallskip	
(b)\ 
Let us say that a separating map $T\colon L^p(\M)\to L^p(\N)$ is $2$-separating 
if the tensor extension $I_{S^p_2}\otimes T\colon L^p(M_2(\M))\to L^p(M_2(\N))$ is separating. 
Combining Proposition \ref{sep} and the 
argument in the proof of \cite[Proposition 3.2]{JRS}, we obtain
that $T\colon L^p(\M)\to L^p(\N)$ is $2$-separating if and only if the Jordan homomorphism in
the Yeadon factorisation of $T$ is multiplicative. This fact was 
also observed in \cite[Theorem 3.4]{HRW}. 
\end{remark}

\begin{remark}
Here we discuss the commutative case.
Let $(\Omega,\F,\mu)$ and $(\Omega',\F',\mu')$ be two 
$\sigma$-finite measure spaces. Let $\widetilde{\F}$ be the 
set of classes of $\F$ modulo the null sets (i.e. $E\in\F$ 
such that $\mu(E)=0$). We identify any element of $\F$ with 
its class in $\widetilde{\F}$.
We define $\widetilde{\F'}$ similarly.

Recall that a 
{\bf regular set morphism} (RSM) from $(\Omega,\F,\mu)$ into 
$(\Omega',\F',\mu')$ is a map
$\phi\colon \widetilde{\F}\to\widetilde{\F'}$ satisfying 
the following two properties:
\begin{itemize}
\item [(i)] For any $E_1,E_2\in \F$, $E_1\bigcap E_2=\emptyset
\,\Rightarrow\, \phi(E_1)\bigcap \phi(E_2)=\emptyset$.
\item [(ii)] For any sequence $(E_n)_{n\geq 1}$
of pairwise disjoint sets in $\F$, 
$\phi\bigl(\bigcup_{n\geq 1} E_n\bigr)
=\bigcup_{n\geq 1}\phi( E_n)$.
\end{itemize} 

Following \cite{Lam},
Kan \cite[Theorem 4.1]{K} showed that a bounded operator $T\colon
L^p(\Omega)\to L^p(\Omega')$ is separating if and only if there exist a 
measurable function $h$ and a regular set morphism $\phi$ 
from $(\Omega,\F,\mu)$ into 
$(\Omega',\F',\mu')$
such that for every set of finite measure $E$, we have
$$
T(\chi_E)=h\cdot \chi_{\phi(E)}.
$$
(See \cite{K} and \cite[Chapter 3]{FJ1} for more on this factorization property.)

There is a well-known correspondence between RSM and 
normal $*$-homomorphisms on $L^\infty$-spaces. Namely, let 
$\pi\colon L^\infty(\Omega,\F,\mu)\to L^\infty(\Omega',\F',\mu')$ 
be a normal $\ast$-homomorphism. Then
for any $E\in\F$, $\pi(\chi_E)$ is a projection, hence an indicator function. 
We may therefore define 
$\phi\colon \widetilde{\F}\to\widetilde{\F'}$  by $\pi(\chi_E)=\chi_{\phi(E)}$. 
It is easy to check that $\phi$ is a RSM. 
It turns out that any regular set morphism is of this form. 
Indeed let $\phi\colon \widetilde{\F}\to\widetilde{\F'}$ be a RSM. 
For any $g\in 
L^1(\Omega',\F',\mu')$, define $\nu_g\colon\F\to\Cdb$ by
$$
\nu_g(E)=\int_{\phi(E)} g(t)\,d\mu'(t),\qquad\hbox{for any}\  E\in\F.
$$
By (ii) and Lebegue's theorem, $\nu_g$ is a complex measure, 
whose total variation is less than or equal to $\norm{g}_1$.
By (i), $\varphi(\emptyset)=\emptyset$ hence 
$\nu$ is absolutely continuous with respect to $\mu$. 
Hence by the Radon-Nikodym theorem, there exists 
a necessarily unique $h_g\in 
L^1(\Omega,\F,\mu)$ such that
$$
\nu_g(E)=\int_{E} h_g(s)\,d\mu(s),\qquad\hbox{for any}\  E\in\F.
$$
Moreover $\norm{h_g}_1\leq\norm{g}_1$.
It is clear that the mapping $u\colon 
L^1(\Omega',\F',\mu')\to L^1(\Omega,\F,\mu)$
defined by $u(g)=h_g$ is linear. The above estimate 
shows that $u$ is contractive. 
Set
$$
\pi=u^*\colon L^\infty(\Omega,\F,\mu)
\longrightarrow L^\infty(\Omega',\F',\mu').
$$
By construction, $\pi$ is $w^*$-continuous. Further it is easy to check that
$\pi$ is a $*$-homomorphism and that $\pi(\chi_E)=\chi_{\phi(E)}$ 
for any $E\in\F$.

We finally note that all Jordan homomorphisms on $L^\infty$-spaces are $*$-homomorphisms. 
Thereby, through the aforementioned correspondence, Proposition \ref{sep} 
reduces to Kan's theorem in the case when $\M$ and $\N$ are commutative. 
\end{remark}

\begin{remark}\label{Injective}
\

(a)\
In general, separating operators may not be one-to-one (contrary to
isometries). We observe however that if a bounded operator
$T\colon L^p(\M)\to L^p(\N)$ is separating, with 
Yeadon triple $(w,B,J)$, then $T$ is one-to-one if and only 
if $J$ is a Jordan monomorphism. 

Indeed if $T$ is one-to-one, then $J(e)=s\bigl(\vert T(e)\vert\bigr)\not=0$ for
any $e\in\E\setminus\{0\}$. This implies that for any pairwise disjoint non zero
$e_1,\ldots, e_m$ in $\E$ and any $\alpha_1,\ldots,\alpha_m$ in $\Cdb$, we have
$$
\Bignorm{\sum_k \alpha_k J(e_k)}\,=\,\max\bigl\{\vert\alpha_k\vert\,
:\, k=1,\ldots,m\bigr\}\,=\,
\Bignorm{\sum_k \alpha_k e_k}.
$$
Hence the restriction of $J$ to $\E$ is an isometry. This readily implies
that $J$ is an isometry, and hence $J$ is one-to-one. 

Assume conversely that $J$ is one-to-one
and let $x\in L^p(M)$ such that $T(x)=0$.
Let $x=u\vert x\vert$ be the polar decomposition of $x$ and for any
integer $n\geq 1$, let $p_n=\chi_{[-n,n]}(\vert x\vert)$. Then consdider
$x'_n = x p_n$ and $x''_n= x(1-p_n)$. 
We have $x=x'_n+x''_n$ hence 
$T(x'_n)+T(x''_n)=0$.
Further we have
$$
{x'_n}^* x''_n = p_n\vert x \vert u^* u \vert x \vert (1-p_n) = 
p_n\vert x \vert^2 (1-p_n)=0,
$$
whereas $x'_n {x''_n}^* = xp_n(1-p_n) x^*=0$. Hence $x'_n$ and $x''_n$ are disjoint.
Since $T$  is separating this implies that 
$T(x'_n)$ and $T(x''_n)$ are disjoint. Since these elements are opposite to each other, 
this implies that $T(x'_n)=0$.
For any $n\geq 1$, $x'_n\in L^p(M)\cap M$ hence $T(x'_n)=wBJ(x'_n)$.
Since $w^*w=s(B)=J(1)$, this implies that
$J(x'_n)=0$, hence $x'_n=0$ by our assumption. Since $x'_n\to x$ in $L^p(M)$,
we deduce that $x=0$. This shows that $T$ is one-to-one.

\smallskip 
(b)\ 
We also observe that a separating operator
$T\colon L^p(\M)\to L^p(\N)$ with 
Yeadon triple $(w,B,J)$ is positive if and only 
$w$ is positive (if and only if $w$ is a projection). 
The verification is left to the reader.
\end{remark}

The following theorem shows that  separating  bounded 
operators are $\ell^1$-bounded. The converse does not hold true, this 
can be easily seen on commutative $L^p$-spaces (see Remark \ref{Comm}).

\begin{theorem}\label{main0}
Suppose that $T\colon L^p(\mathcal{M})\to L^p(\mathcal{N})$ is a
separating  bounded operator, with $1\leq p<\infty$. 
Then $T$ is $\ell^1$-bounded 
and $\norm{T}_{\ell^1} = \norm{T}$.
\end{theorem}

\begin{proof}
We apply Proposition \ref{sep}. We let 
$(w,B,J)$ be the Yeadon triple of the operator $T$. Next according
to (\ref{pisigma}), we let
$
J=\begin{pmatrix}
\pi&0\\
0&\sigma
\end{pmatrix}
$ 
be a central decomposition of $J$ and we let
$e,f$ be the central projections such that $\pi=J(\,\cdotp)e$ 
and $\sigma=J(\,\cdotp)f$.

Let $T_1\colon L^p(\M)\to L^p(\N)$ and 
$T_2\colon L^p(\M)\to L^p(\N)$ be defined by 
$$
T_1(x)=T(x)e
\qquad\hbox{and}\qquad
T_2(x)=T(x)f
$$
for any $x\in L^p(\M)$. 
Then, $T=T_1+T_2$ and for all $u,v\in L^p(\M)$, we have
\begin{equation}\label{p-sum}
\|T_1(u)+T_2(v)\|_p^p=\|T_1(u)\|_p^p+\|T_2(v)\|_p^p.
\end{equation}
Let $(x_n)_{n\geq 1}\in L^p(\mathcal{M};\ell^1)$, with 
$\|(x_n)\|_{L^p(\mathcal{M};\ell^1)}<1$. By Lemma \ref{2.6}, 
there exist factorizations $x_n=a_nb_n$, $n\geq 1$,
with $a_n,b_n\in L^{2p}(\M)$, such that
$$
\Bignorm{\sum_{n=1}^{\infty}
a_na_n^*}_p <1
\qquad\hbox{and}\qquad
\Bignorm{\sum_{n=1}^{\infty} b_n^*b_n}_p <1.
$$

The identity (\ref{NormY}) and its proof show that for
any $a\in \M\cap L^{2p}(\M)$, we have
$$
\|B^{\frac12}\pi(a)\|_{2p}^{2p} =\tau_{\tiny\N}
(B^p\pi(\lvert a \rvert^{2p}))
\leq\tau_{\tiny\N}(B^pJ(\lvert a\rvert^{2p}))=
\|T(\lvert a\rvert^2)\|^p_p.
$$
Similarly, $\|B^{\frac12}\sigma(a)\|_{2p}^{2p}
\leq\|T(\lvert a \rvert^2)\|_p^p$. 
Hence there exist two bounded operators 
$$
S_1\colon L^{2p}(\mathcal{M})\longrightarrow L^{2p}(\mathcal{N})
\qquad\hbox{and}\qquad
S_2\colon L^{2p}(\mathcal{M})\longrightarrow L^{2p}(\mathcal{N})
$$ 
such that
$$
S_1(a)=B^{\frac{1}{2}}\pi(a)
\qquad\hbox{and}\qquad
S_2(a)=B^{\frac{1}{2}}\sigma(a),
$$
for all $a$ in $\mathcal{M}\cap L^{2p}(\mathcal{M})$.
It is clear from above that  
for any $a$ and $b$ in $L^{2p}(\mathcal{M})$ 
we have 
$$
S_1(a)S_1(b)=w^\ast T_1(ab),
\quad
S_2(a)S_2(b)=w^\ast T_2(ba)
\quad\hbox{and}\quad
S_1(a)S_2(b)=S_2(b)S_1(a)=0.
$$
We will use this repeatidly  in the rest of the proof.

For any $n\geq 1$, we have
$$
T(x_n)=T(a_nb_n)=T_1(a_nb_n)+T_2(a_nb_n)
= wS_1(a_n)S_1(b_n) +  wS_2(b_n)S_2(a_n).
$$
Hence $T(x_n)=c_nd_n$, with
$$
c_n =w\left(S_1(a_n)+S_2(b_n)\right)
\qquad\hbox{and}\qquad
d_n=S_1(b_n)+S_2(a_n). 
$$
With a similar computation, we obtain
\begin{align*}
c_nc_n^* & = w\left(S_1(a_n)+S_2(b_n)\right)
\left(S_1(a_n^*)+S_2(b_n^*)\right)w^*\\
& = 
\left(T_1(a_na_n^*)+T_2(b_n^*b_n)\right)w^*.
\end{align*}

Let $N\geq 1$ and set $u_N=\sum_{n=1}^N a_na_n^*$ and
$v_N=\sum_{n=1}^N b_n^* b_n$. Summing up we obtain 
$$
\sum_{n=1}^N c_nc_n^* = \left(T_1(u_N)+T_2(v_N)\right)w^*
$$
Appealing to (\ref{p-sum}), we deduce that 
$$
\Bignorm{\sum_{n=1}^N c_nc_n^*}^p_p\leq 
\|T_1(u_N)\|_p^p+\|T_2(v_N)\|_p^p.
$$
Similarly, 
$$
\Bignorm{\sum_{n=1}^N d_n^*d_n}^p_p\leq
\|T_1(v_N)\|_p^p+\|T_2(u_N)\|_p^p.
$$
Consequently,
$$
\Bignorm{\sum_{n=1}^N c_nc_n^*}^p_p
\Bignorm{\sum_{n=1}^N d_n^*d_n}^p_p
\leq
\left(\|T_1(u_N)\|_p^p+\|T_2(v_N)\|_p^p\right)
\left(\|T_1(v_N)\|_p^p+\|T_2(u_N)\|_p^p\right).
$$
Let $\alpha=\|T_1(u_N)\|_p^p$ and $\beta=\|T_1(v_N)\|_p^p$. Since, 
$$
\|T_1(u_N)\|_p^p+\|T_2(u_N)\|_p^p=\|T(u_N)\|_p^p\leq\|T\|^p\norm{u_N}^p_p,
$$
by (\ref{p-sum}), and $\norm{u_N}_p<1$, 
we have $\|T_2(u)\|_p^p\leq\|T\|^p-\alpha$. 
Similarly, $\|T_2(v)\|_p^p\leq\|T\|^p-\beta.$ 
Therefore,
\begin{align*}
\left(\|T_1(u_N)\|_p^p+\|T_2(v_N)\|_p^p\right) 
\left(\|T_1(v_N)\|_p^p+\|T_2(u_N)\|_p^p\right)
&\leq (\alpha+(\|T\|^p-\beta))(\beta+(\|T\|^p-\alpha))\\
&=\left(\|T\|^p-(\beta-\alpha)\right) \left(\|T\|^p+(\beta-\alpha)\right)\\
&=(\|T\|^{2p}-\left(\beta-\alpha\right)^2)\leq \|T\|^{2p},
\end{align*}
and hence
$$
\Bignorm{\sum_{n=1}^N c_nc_n^*}_p^{\frac12}
\Bignorm{\sum_{n=1}^N d_n^*d_n}_p^{\frac12}
\leq\norm{T}.
$$
This implies that $(T(x_n))_{n\geq 1}$ belongs to $L^p(\N;\ell^1)$
and that its norm in $L^p(\N;\ell^1)$ is less than or equal to 
$\norm{T}$.
This yields the boundedness of $T\otimes I_{\ell^1}$, as
well as the equality 
$\norm{T}_{\ell^1} = \norm{T}$.
\end{proof}

\section{Isometries on $L^2$-spaces with a Yeadon type factorization}\label{sec5}

As it is outlined in Section \ref{sec3}, the crucial property that allowed Yeadon 
to describe isometries between noncommutative $L^p$-spaces, $p\not=2$, is that they are separating. 
To show that every isometry is indeed separating 
he relied on the property that when $1\leq p\neq2<\infty$, 
the equality condition in Clarkson's inequality,
$\displaystyle{\|a+b\|_p+\|a-b\|_p=2(\|a\|_p+\|b\|_p)}$, holds true if and only if $a$ and $b$ are disjoint. 
However, this equality always holds true when $p=2$ and this is why the study of disjointness
on $L^2$-spaces requires a different approach.
This is the purpose of Lemma \ref{dinq} below and as a consequence, 
we will characterize isometries between noncommutative $L^2$-spaces which 
admit a Yeadon type factorization.

\begin{lemma}\label{dinq}
Suppose that $a,b\in L^2(\mathcal{M})$. Then, $a$ and $b$ are disjoint if and only if we have 
$$
\|(a,b)\|_{L^2(\tiny\M;\ell_2^1)}\leq\left(\|a\|^2_2+\|b\|_2^2\right)^{1/2}.
$$
\end{lemma}

\begin{proof}
First suppose that for disjoint elements $a,b\in L^2(\mathcal{M})$ the polar decompositions are 
given by $a=u\lvert a\rvert$ and $b=v\rvert b\lvert$. Define
$$
a_1 =u\lvert a\rvert^{1/2},\quad
a_2 =\lvert a\rvert^{1/2},\quad 
b_1=v\lvert b\rvert^{1/2}\quad\hbox{and}\quad
b_2=\lvert b\rvert^{1/2}.
$$ 
These elements belong
to $L^4(\M)$ and we have $a=a_1a_2$ and $b=b_1b_2$.
Further we have 
$$
a_1a_1^*+b_1b_1^*=u\lvert a\rvert u^*+v\lvert b\rvert 
v^*
\qquad\hbox{and}\qquad
a_2^*a_2+b_2^*b_2=\lvert a\rvert+\lvert b\rvert. 
$$
Consequently,
$$
\|(a,b)\|_{L^2(\tiny\M;\ell_2^1)}^2\leq\|u\lvert a\rvert 
u^*+v\lvert b\rvert v^*\|_2\,\|\lvert a\rvert+\lvert b\rvert\|_2.
$$
Now, since $a$ and $b$ are disjoint we have that 
$\lvert a\rvert\lvert b\rvert=0$, by Lemma \ref{disjointness}, and so
\begin{equation*}
\begin{split}
\|\lvert a\rvert+\lvert b\rvert\|_2^2 &= 
\|\lvert a\rvert \|_2^2 + \|\lvert b\rvert\|_2^2 + 2 \tau_{\tiny_{\M}}(\lvert a\rvert\lvert b\rvert)\\
& =\| a \|^2_2+\|  b \|^2_2.
\end{split}
\end{equation*}
Similarly, since $a$ and $b$ are disjoint, we have $ua^*bv^*=0$, and so
\begin{align*}
\|u\lvert a\rvert u^*+v\lvert b\rvert v^*\|_2^2&=\|u\lvert a\rvert u^*\|^2_2+\|v\lvert b\rvert v^*\|^2_2+2 \tau_{\tiny_{\M}}\left(u\lvert a\rvert u^* v\lvert b\rvert v^*\right)\\
&=\|u\lvert a\rvert u^*\|^2_2+\|v\lvert b\rvert v^*\|^2_2+2  \tau_{\tiny_{\M}}\left(ua^*bv^*\right)\\
&=\|u\lvert a\rvert u^*\|^2_2+\|v\lvert b\rvert v^*\|^2_2\\
&\leq\ \| a\|^2_2+\| b\|^2_2.
\end{align*}
This implies that $\|(a,b)\|_{L^2(\tiny\M;\ell_2^1)}\leq\left(\|a\|^2_2+\|b\|^2_2\right)^{1/2}$.

Conversely, suppose that $a,b\in L^2(\mathcal{M})$ satisfy
$\|(a,b)\|_{L^2(\mathcal{M};\ell^1_2)}\leq\left(\|a\|_2^2+\|b\|_2^2\right)^{1/2}.$
Let $(\epsilon_k)_k$ be a sequence of positive real numbers with $\lim_k\epsilon_k=0$.
By Lemma \ref{2.6}, 
\begin{align*}
\|(a,b)\|_{L^2(\tiny\M;\ell^1)}=\inf\left\{\|uu^*+ww^*\|_2^{\frac12}\|v^*v+z^*z\|_2^{\frac12}\right\},
\end{align*}
where the infimum is taken over all factorizations $a=uv$ and $b=wz$, with $u,v,w,z\in L^4(\M)$.
Thus for any $k\geq 1$, we can find $u_k,v_k,w_k, z_k\in L^4(\M)$ such that $a=u_kv_k$, $b=w_kz_k$,
\begin{equation}\label{Bdd1}
\| u_ku_k^*+w_kw_k^*\|_2^{\frac12}\leq
(1+\epsilon_k)^{\frac18}\left(\|a\|_2^2+\|b\|_2^2\right)^{\frac14},
\end{equation}
and 
\begin{equation}\label{Bdd2}
\| v_k^*v_k+z_k^*z_k\|_2^{\frac12}\leq(1+\epsilon_k)^{\frac18}\left(\|a\|_2^2+\|b\|_2^2\right)^{\frac14}.
\end{equation}
This implies that 
\begin{equation}\label{eq1}
\|u_ku_k^*\|_2^2+\|w_kw_k^*\|_2^2+2 \tau_{\tiny_{\M}}\left(u_ku_k^*w_kw_k^*\right)\leq
\left(1+\epsilon_k\right)^{\frac12}\left(\|u_kv_k\|_2^2+\|w_kz_k\|_2^2\right)
\end{equation}
and 
\begin{equation}\label{eq2}
\|v_k^*v_k\|_2^2+\|z_k^*z_k\|_2^2+2 \tau_{\tiny_{\M}}
\left(v_k^*v_kz_k^*z_k\right)\leq\left(1+\epsilon_k\right)^{\frac12}\left(\|u_kv_k\|_2^2+\|w_kz_k\|_2^2\right).
\end{equation}

We claim that  
\begin{equation}\label{CS}
\left(\|u_kv_k\|_2^2+\|w_kz_k\|_2^2\right)^2\leq\bigl(\norm{u_k^*u_k}_2^2 + \norm{w_k^*w_k}_2^2\bigr)
\bigl(\norm{v_k v_k^*}_2^2 + \norm{z_k z_k^*}_2^2\bigr).
\end{equation}
Indeed
$$
\|u_kv_k\|_2^2 = \tau_{\tiny\M}\bigl((u_kv_k)^*(u_kv_k)\bigr) = \tau_{\tiny\M}\bigl(v_kv_k^*u_k^*u_k\bigr) = 
\,\bigl\langle u_k^*u_k, v_k v_k^*\bigr\rangle_{L^2(\tiny\M)}
$$
and similarly, $\|w_kz_k\|_2^2 = \langle w_k^*w_k, z_k z_k^*\bigr\rangle_{L^2(\tiny\M)}$.
Hence (\ref{CS}) follows by applying 
the
Cauchy-Schwarz inequality in the Hilbertian direct sum $L^2(\M)
\mathop{\oplus}\limits^{2} L^2(\M)$.

Multiplying inequalities (\ref{eq1}) and (\ref{eq2})
and using 
the fact that $\tau_{\tiny_{\M}}
\left(v_k^*v_kz_k^*z_k\right)\geq 0$, we obtain that
\begin{equation}\label{LHS}
\bigl(\|u_ku_k^*\|_2^2+\|w_kw_k^*\|_2^2+2 \tau_{\tiny_{\M}}\left(u_ku_k^*w_kw_k^*\right)\bigr)
\bigl(\|v_k^*v_k\|_2^2+\|z_k^*z_k\|_2^2\bigr)
\end{equation}
is less than or equal to $(1+\epsilon_k)\left(\|u_kv_k\|_2^2+\|w_kz_k\|_2^2\right)^2$. 
Now using (\ref{CS}) we deduce that (\ref{LHS}) is less than or equal to
$$
(1+\epsilon_k)\bigl(\norm{u_k^*u_k}_2^2 + \norm{w_k^*w_k}_2^2\bigr)
\bigl(\norm{v_k v_k^*}_2^2 + \norm{z_k z_k^*}_2^2\bigr).
$$
Now we observe that $\|u_ku_k^*\|_2^2 =\norm{u_k}^4_4 =\|u_k^*u_k\|_2^2$ and similarly for $w_k$, $v_k$ and $z_k$. 
Hence the above inequality reads
$$
\bigl(\norm{u_k}^4_4 +\norm{w_k}^4_4+2 \tau_{\tiny_{\M}}\left(u_ku_k^*w_kw_k^*\right)\bigr)
\leq (1+\epsilon_k)\bigl(\norm{u_k}^4_4 + \norm{w_k}^4_4\bigr).
$$
This yields
$$
\tau_{\tiny_{\M}}\left(u_ku_k^*w_kw_k^*\right)
\leq\epsilon_k \left(\frac{\norm{u_k}^4_4+\norm{w_k}^4_4}{2}\right).
$$
It follows from (\ref{Bdd1}) that
$(\norm{u_k}_4)_k$ and $(\norm{w_k}_4)_k$ are bounded sequences.
Hence we have that $\tau_{\tiny_{\M}}(u_ku_k^*w_kw_k^*)\to0$ as $k\to\infty$. Writing
$$
\tau_{\tiny_{\M}}(u_ku_k^*w_kw_k^*) = \tau_{\tiny_{\M}}(u_k^*w_kw_k^*u_k)
= \tau_{\tiny_{\M}}((u_k^*w_k)(u_k^*w_k)^*) = \norm{u_k^*w_k}_2^2,
$$
we deduce that $\norm{u_k^*w_k}_2\to0$ as $k\to\infty$.

We have $a^*b=v_k^*u_k^*w_kz_k$, hence 
$$
\|a^*b\|_1\leq\|v_k\|_4\|u_k^*w_k\|_2\|z_k\|_4.
$$
By (\ref{Bdd2}),
$(\norm{v_k}_4)_k$ and $(\norm{z_k}_4)_k$ are bounded sequences, hence the right hand side 
in the above inequality tends to $0$ as $k\to\infty$. 
We deduce that $a^*b=0$.

Finally using $\tau_{\tiny_{\M}}\left(v_k^*v_k z_k^*z_k\right)$
instead of $\tau_{\tiny_{\M}}\left(u_ku_k^*w_kw_k^*\right)$, we show 
as well that $ab^*=0$
and therefore, $a$ and $b$ are disjoint.
\end{proof}

\begin{theorem}\label{main1}
For a linear isometry $T\colon
L^2(\mathcal{M})\to L^2(\mathcal{N})$, the following statements are equivalent:
\begin{enumerate}[(i)]
\item $T$ has a Yeadon type factorization.
\item $T$ is $\ell^1_2$-contractive.
\item $T$ is $\ell^1$-contractive.
\end{enumerate} 
\end{theorem}

\begin{proof}
In the light of Proposition \ref{sep} and  Theorem \ref{main0}, 
we only need to establish that if (ii) holds true, then
$T$ is separating.

Suppose that $T$ is $\ell^1_2$-contractive.
Let $a,b\in L^2(\mathcal{M})$ be disjoint elements. By Lemma \ref{dinq},
$$
\|(Ta,Tb)\|_{L^2(\tiny\N;\ell_2^1)} \leq\|(a,b)\|_{L^2(\tiny\M;\ell^1_2)}
\leq (\| a\|_2^2+\|b\|_2^2)^{\frac12}.
$$
Since $T$ is an isometry we have $\norm{T(a)}=\norm{a}$ and 
$\norm{T(b)}=\norm{b}$ and hence
$$
\|(Ta,Tb)\|_{L^2(\tiny\N;\ell_2^1)} \leq 
(\| Ta\|_2^2+\| Tb\|_2^2)^{\frac12}.
$$
By Lemma \ref{dinq} again, this implies 
that $Ta$ and $Tb$ are disjoint.
Hence $T$ is separating.
\end{proof}

\begin{remark}
\

(a)\ As mentioned in Remark \ref{Comm}, 
when $\mathcal{M}=L^\infty(\Omega)$ and $\mathcal{N}=L^\infty(\Omega')$ are commutative,  
a bounded operator $T\colon L^2(\Omega)\to L^2(\Omega^\prime)$ is $\ell^1$-contractive 
if and only if $T$ is regular, with $\|T\|_{reg}\leq1$. Hence, Theorem  
\ref{main1} implies that in the commutative case, 
an isometry $T\colon L^2(\Omega)\to L^2(\Omega^\prime)$ is 
separating if and only if $\|T\|_{reg}\leq1$.
This result is implicit in \cite{Pel}.

\smallskip	
(b)\  Let $T\colon L^p(\M)\to L^p(\N)$, $1\leq p<\infty$, be a separating isometry,
with the Yeadon triple $(w,B,J)$. We show in \cite{LMZ} that when $p=2$ and $\M$ and $\N$ are hyperfinite, then
$J$ is multiplicative (equivalently, $J$ is a $*$-homomorphism)
if and only if $T$ is completely regular with $\|T\|_{\text{reg}}=1$. 
This is an $L^2$-analog of \cite[Theorem 3.1]{JRS} which says that for $1\leq p\neq 2<\infty$, 
$J$ is multiplicative if and only if $T$ is a complete isometry. 
\end{remark}

In \cite{B}, Broise showed that every bijective positive isometry 
between noncommutative $L^2$-spaces 
associated with semifinite factors admits a Yeadon type factorization. 
Using Proposition \ref{sep}, one can 
actually obtain the following more general statement.

\begin{corollary}\label{bor}
Suppose that $T\colon L^2(\mathcal{M})\to L^2(\mathcal{N})$ is a positive isometry. 
Then $T$ admits a Yeadon type factorization.
\end{corollary}

\begin{proof}
Let $a,b\in L^2(\M)$ be positive elements, with $ab=0$. They 
are orthogonal and isometries preserve orthogonality, hence
$T(a)$ and $T(b)$ are orthogonal. Since $T(a)$ and $T(b)$
are positive, Remark \ref{Ortho} ensures that  $T(a)$ and $T(b)$ are
disjoint. 

By Remark \ref{Restrict} (a) and Proposition \ref{sep}, the above shows that 
$T$ admits a Yeadon type factorization.
\end{proof}

\section{positivity and $\ell^1$-contractivity}\label{sec6}

For any $n\geq2$, we let $S^p_n=S^p(\ell_2^n)$
and we let $S^p_n[L^p(\M)]$ be the space $S^p_n\otimes L^p(\M)$ equipped with the norm and the
partial order coming from its
identification with the space $L^p(M_n(\M))$, see Section \ref{sec2}.

We say that a bounded operator $T\colon
L^p(\mathcal{M})\to L^p(\mathcal{N})$, $1\leq p < \infty$, is $n$-positive if
$$
I_{S^p_n}\otimes T\colon S^p_n[L^p(\mathcal{M})]\longrightarrow S^p_n[L^p(\mathcal{N})]
$$
is positive. We say that $T$ is completely positive if it is $n$-positive for all $n\geq 1$.

\begin{proposition}\label{6.2}
Suppose that $T\colon L^p(\M)\to L^p(\N)$ is a $2$-positive contraction, then $T$ is $\ell^1$-contractive.
\end{proposition}

\begin{proof}
Let $T\colon L^p(\M)\to L^p(\N)$ be a $2$-positive contraction and let
$(x_n)_{n\geq1}$ be a sequence in $L^p(\M)$ such that $\|(x_n)_{n\geq 1}\|_{L^p(\tiny\M;\ell^1)}<1$. 
According to Lemma \ref{2.6}, we may choose
sequences $(a_n)_{n\geq1}$ and $(b_n)_{n\geq1}$ in $L^{2p}(\M)$ such that 
$x_n=a_n b_n$ for any $n\geq 1$,  
\begin{equation}\label{Norm<1}
\Bignorm{\sum_{n=1}^\infty a_n a_n^*}_p<1
\qquad\hbox{and}\qquad
\Bignorm{\sum_{n=1}^\infty b_n^*b_n}_p <1. 
\end{equation}
For any $n\geq 1$, let
$$
z_n:=\begin{pmatrix}
a_na_n^*&a_nb_n\\
b_n^* a_n^*&b_n^*b_n
\end{pmatrix}
$$ 
in $S^p_2[L^p(\M)]$. Then 
$z_n=\begin{pmatrix}
a_n&0\\
b_n^*&0
\end{pmatrix}
\begin{pmatrix}
a_n^*&b_n\\
0&0
\end{pmatrix},
$ hence $z_n\geq 0$.
Therefore by the $2$-positivity of $T$, 
$$
(I_{S_2^p}\otimes T)(z_n)=\begin{pmatrix}
T(a_n a_n^*)& T(a_n b_n)\\
T(b_n^* a_n^*)& T(b_n^* b_n)
\end{pmatrix}\,\geq0.
$$
Consider the positive square root 
$$
\begin{pmatrix}
\alpha_n&\beta_n\\
\beta_n^*&\delta_n
\end{pmatrix}:=\bigl((I_{S_2^p}\otimes T)(z_n)\bigr)^{\frac{1}{2}}.
$$
Then $\alpha_n,\beta_n,\delta_n$ belong to $L^{2p}(\N)$, we have 
$\alpha_n\geq0$, $\delta_n\geq0$, and
\begin{align*}
T(a_{n}a_{n}^*) & =\alpha_{n}^2+ \beta_{n}\beta_n^*,\\
T(b_{n}^*b_{n}) & =\beta_n^*\beta_n +\delta_n^2,\\
T(a_nb_n) & = \alpha_n\beta_n+\beta_n\delta_n.
\end{align*}
Using the third equation above and Junge's definition of $L^p(\N;\ell^1)$, we get that
$$
\|\left(T(a_nb_n)\right)_{n\geq 1}\|_{L^p(\tiny\N;\ell^1)}
\leq \Bignorm{\sum_{n=1}^{\infty}\alpha_n^2+\sum_{n=1}^{\infty}\beta_n\beta_n^*}_p^{\frac{1}{2}} \,
\Bignorm{\sum_{n=1}^{\infty} \beta_n^*\beta_n+\sum_{n=1}^{\infty}\delta_n^2}_p^{\frac{1}{2}}.
$$
(The convergence of the series are justified by the next lines.)

We can now apply the first two equations and (\ref{Norm<1}) to deduce that
\begin{align*}
\|\left(T(x_n)\right)_{n\geq 1}\|_{L^p(\tiny\N;\ell^1)}&\leq
\Bignorm{\sum_{n=1}^{\infty} T(a_na_n^*)}_p^{\frac{1}{2}}
\Bignorm{\sum_{n=1}^{\infty} T(b_n^*b_n)}_p^{\frac{1}{2}}\\
&\leq \Bignorm{T\Bigl(\sum_{n=1}^{\infty} a_n a_n^*\Bigr)}_p^{\frac{1}{2}} 
\Bignorm{T\Bigl(\sum_{n=1}^{\infty} b_n^* b_n\Bigr)}_p^{\frac{1}{2}}\\
&\leq\|T\|\Bignorm{\sum_{n=1}^\infty a_n a_n^*}_p^\frac12
\Bignorm{\sum_{n=1}^\infty b_n^* b_n}_p^\frac12\\
& < 1.
\end{align*}
This shows that $T$ is $\ell^1$-contractive.
\end{proof}

\begin{remark}
\

(a)\ An obvious consequence of Proposition \ref{6.2} is that if $T$ is a completely positive contraction then 
it is $\ell^1$-contractive.
	
\smallskip	
(b)\ 
Let $\N^{op}$ be the opposite 
von Neumann algebra of $\N$ and let $I_{op}\colon L^p(\N)\to L^p(\N^{op})$ denote the identity map. 
We say that $T\colon L^p(\M)\to L^p(\N)$ is $2$-copositive if 
the operator $I_{op}\circ T\colon L^p(\M)\to L^p(\N^{op})$ 
is $2$-positive. It is easy to check that
$$
L^p(\N;\ell^1)=L^p(\N^{op},\ell^1)
\qquad\text{isometrically}.
$$
Therefore, Proposition \ref{6.2} implies
that any contractive $2$-copositive map $L^p(\M)\to L^p(\N)$ is $\ell^1$-contractive. 
It therefore follows that if a positive map $L^p(\M)\to L^p(\N)$ can be written as a convex combination
of a contractive $2$-positive map and a contractive $2$-copositive map,
then it is $\ell^1$-contractive.
\end{remark}

We do not know if any positive contraction 
is $\ell^1$-contractive, however we
show below that positive operators are $\ell^1$-bounded.

\begin{proposition}
Let $T\colon L^p(\M)\to L^p(\N)$ be a bounded operator. 
If $T$ is positive, then $T$ is $\ell^1$-bounded, with $\|T\|_{\ell^1}\leq4\|T\|$.
\end{proposition}

\begin{proof}
As in the proof of Proposition \ref{6.2}, let
$(x_n)_{n\geq1}$ be a sequence in $L^p(\M)$ such that $\|(x_n)_{n\geq 1}\|_{L^p(\tiny\M;\ell^1)}<1$, and let  
$(a_n)_{n\geq1}$ and $(b_n)_{n\geq1}$ in $L^{2p}(\M)$ such that 
$x_n=a_n b_n$ for any $n\geq 1$ and (\ref{Norm<1}) holds.
 
For any $n\geq 1$, we use the polarization identity,
$$
a_nb_n=\frac{1}{4}\sum_{k=0}^3 (-i)^k(a_n^*+i^kb_n)^*(a_n^*+i^kb_n).
$$ 
For $0\leq k\leq3$ and $n\geq1$, let $y_n^k:= (a_n^*+i^kb_n)^*(a_n^*+i^kb_n)$. 
Then $y_n^k\geq 0$ hence $T(y_n^k)\geq 0$. This implies, by Lemma \ref{Xu}, that for any $k$, 
$$
\norm{(T(y_n^k))_n}_{L^p(\tiny\N;\ell^1)} = \Bignorm{
\sum_{n=1}^\infty T(y_n^k)}_p\leq\norm{T}\Bignorm{
\sum_{n=1}^\infty y_n^k}_p.
$$
(The convergence of the series are justified by the next lines.)

Moreover
\begin{align*}
\Bignorm{\sum_{n=1}^\infty y_n^k}_p
& = \Bignorm{\sum_{n=1}^\infty E_{n1}\otimes (a_n^*+i^kb_n)}_{L^{2p}(\tiny\M;\ell^2_c)}^2\\
& \leq \biggl(\Bignorm{\sum_{n=1}^\infty E_{n1}\otimes a_n^*}_{L^{2p}(\tiny\M;\ell^2_c)}\, +\,
\Bignorm{\sum_{n=1}^\infty E_{n1}\otimes b_n}_{L^{2p}(\tiny\M;\ell^2_c)}\biggr)^2\\
&\leq 4
\end{align*}
by (\ref{SF}) and (\ref{Norm<1}).

Since
$$
T(x_n)=\frac{1}{4}\sum_{k=0}^3(-i)^kT(y_n^k),
$$
we deduce that
$$
\|(T(x_n))_n\|_{L^p(\tiny\N;\ell^1)}
\leq\frac{1}{4}\sum_{k=0}^3\|(T(y_n^k))_n\|_{L^p(\tiny\N;\ell^1)}\leq 
4\norm{T}.
$$
This shows that $\|T\|_{\ell^1}\leq4\|T\|$.
\end{proof}

\bigskip
\noindent{\bf Acknowledgements.} 
The first named author is supported by the French 
``Investissements d'Avenir" program, 
project ISITE-BFC (contract ANR-15-IDEX-03).
This project was carried out during a visit
of the second named author at the
``Laboratoire de Math\'ema-\-tiques de Besan\c con" (LmB).
She wishes to thank the LmB
for hospitality and support. Finally the authors thank the
referee for his/her suggestions and careful reading.

\nocite{*}

\bibliographystyle{alpha}

\end{document}